\PassOptionsToPackage{dvipsnames}{xcolor}
\newif\ifXcolor\Xcolortrue
\documentclass[a4paper,10pt]{article}
\usepackage[T1]{fontenc}
\usepackage{lmodern}
\usepackage[margin=25mm]{geometry}
\usepackage{amsmath,amssymb,amsthm,bm}
\usepackage{aliascnt}
\usepackage{adjustbox}
\usepackage{authblk}
\usepackage{booktabs}
\usepackage{pgfplots}
\usepackage{siunitx}
\usepackage[numbers,sort&compress]{natbib}
\usepackage[colorlinks,linkcolor=blue!60!black,citecolor=green!50!black,urlcolor=red!50!black]{hyperref}
\usepackage[capitalise,nameinlink]{cleveref}
\pgfplotsset{compat=1.18}
\newcommand\lvvvert{\mathopen{|\mkern-2.4mu|\mkern-2.4mu|}}
\newcommand\rvvvert{\mathclose{|\mkern-2.4mu|\mkern-2.4mu|}}
\newcommand\llbracket{\mathopen{[\mkern-3.1mu[}}
\newcommand\rrbracket{\mathclose{]\mkern-3.1mu]}}
\usetikzlibrary{patterns}
\usetikzlibrary{pgfplots.fillbetween}

\newcommand{\e}{\bm{e}}

\newcommand{\n}{\bm{n}}

\newcommand{\etab}{\bm{\eta}}
\newcommand{\xib}{\bm{\xi}}

\newcommand\sumFKi{\sumFi \sum_{K\mathpunct: F \in \cl K}}
\newcommand\sumFi{\sum_{F \in \F_G^i}}

\numberwithin{equation}{section}
\numberwithin{figure}{section}
\numberwithin{table}{section}

\theoremstyle{plain}
\newtheorem{theorem}{Theorem}[section]

\newaliascnt{lemma}{theorem}
\newtheorem{lemma}[lemma]{Lemma}
\aliascntresetthe{lemma}

\newaliascnt{problem}{theorem}
\newtheorem{problem}[problem]{Problem}
\aliascntresetthe{problem}

\theoremstyle{definition}
\newaliascnt{assumption}{theorem}
\newtheorem{assumption}[assumption]{Assumption}
\aliascntresetthe{assumption}

\theoremstyle{remark}
\newaliascnt{remark}{theorem}
\newtheorem{remark}[remark]{Remark}
\aliascntresetthe{remark}

\crefname{problem}{problem}{problems}
\crefname{assumption}{assumption}{assumptions}

\newcommand\Oh{\Omega^\T}
\newcommand\Ofh{\Oh_f}
\newcommand\Osh{\Oh_s}
\newcommand\Oih{\Oh_i}

\newcommand\symgrad[1]{\epsilon(#1)}
\newcommand\iprod[3][]{(#2,#3)_{#1}}
\newcommand\prodOf{\iprod[\Omega_f]}
\newcommand\prodOs{\iprod[\Omega_s]}
\newcommand\prodGi{\iprod[\Gamma^i]}
\newcommand\dpair[3][]{\langle#2,#3\rangle_{#1}}
\newcommand\pairf{\dpair[\V_f^* \times \V_f]}
\newcommand\pairs{\dpair[\V_s^* \times \V_s]}
\newcommand\pairi{\dpair[H^{-1/2}(\Gamma^i) \times H^{1/2}(\Gamma^i)]}
\newcommand\Norm[2][]{\left\|#2\right\|\ifx#1\empty\else_{#1}\fi}
\newcommand\norm[2][]{\lVert#2\rVert\ifx#1\empty\else_{#1}\fi}

\newcommand\normOf{\norm[\Omega_f]}
\newcommand\normOs{\norm[\Omega_s]}
\newcommand\normOfh{\norm[\Ofh]}
\newcommand\normOsh{\norm[\Osh]}
\newcommand\normOi{\norm[\Omega_i]}
\newcommand\normGi{\norm[\Gamma^i]}
\newcommand\normIf{\norm[\ifluid]}
\newcommand\normIs{\norm[\isolid]}
\newcommand\normiii[1]{\lvvvert#1\rvvvert}
\newcommand\fcdot{\,\cdot\,}
\newcommand\tracenorm{\norm[h,\Gamma^i]}
\newcommand\jump[1]{\llbracket#1\rrbracket}

\newcommand\Jump[1]{g_F^{#1}(\varphi_1, \varphi_2)}
\renewcommand\:{\colon}
\renewcommand\div{\nabla\cdot}
\newcommand\bd{\partial}
\newcommand\cl[2][2mu]{\mkern#1%
  \overline{\mkern-#1 #2\mkern-#1}\mkern#1}%
\newcommand\tsb[1]{_{\textup{#1}}}
\newcommand\tsp[1]{^{\textup{#1}}}
\newcommand\set[1]{\{#1\}}
\newcommand\defset[3][\colon]{\bigl\{#2#1#3\bigr\}}
\newcommand\R{\mathbb R}
\newcommand\N{\mathbb N}
\newcommand\A{\mathcal A}
\newcommand\E{\mathcal E}
\newcommand\F{\mathcal F}
\newcommand\I{\mathcal I}
\renewcommand\H{\mathcal H}
\newcommand\T{\mathcal T}
\newcommand\U{\mathcal U}
\newcommand\V{\mathcal V}
\newcommand\X{\mathcal X}
\renewcommand\P{\mathcal P}
\newcommand\interf{{\Omega_i\tsp{int}}}
\newcommand\ifluid{{\Omega_f\tsp{int}}}
\newcommand\isolid{{\Omega_s\tsp{int}}}
\newcommand\diff{\,\textup{d}}
\newcommand\ds{\diff s}
\newcommand\dt{\diff t}
\newcommand\vres{\bm r_v^{h,n}}
\newcommand\ET{\E^\T\!}
\newcommand\expcT{e^{cT}}
\newcommand\Udiff{\delta U^{h,n}}
\newcommand\pres{\zeta^{h,n}}

\newcommand\tr[1]{\textup{tr}(#1)}

\newcommand\intI{I}
\newcommand\clI{\mkern2mu\overline{\mkern-2mu I}}
\newcommand\ghw[3][]{g^{h,w}\ifx#1\empty\else_{#1}\fi(#2,#3)}
\newcommand\wmax{\hat w}
\newcommand\CF{\theta}

\let\0\emptyset

\title{Numerical Analysis of a Cut Finite Element Approach for
  Fully Eulerian Fluid-Structure Interaction with Fixed Interface}

\author[1]{S. Frei}
\author[2]{T. Knoke}
\author[2]{M. C. Steinbach}
\author[2]{A.-K. Wenske}
\author[2]{T. Wick}
\affil[1]{Department of Mathematics \& Statistics, University of Konstanz,
  Universitätsstr.~10, 78457 Konstanz, Germany\\
  \href{mailto:stefan.frei@uni-konstanz.de}{stefan.frei@uni-konstanz.de}}
\affil[2]{Leibniz Universität Hannover, Institut für Angewandte Mathematik,
  Welfengarten 1, 30167 Hannover, Germany\\
  \href{mailto:knoke@ifam.uni-hannover.de}{knoke@ifam.uni-hannover.de},
  \href{mailto:mcs@ifam.uni-hannover.de}{mcs@ifam.uni-hannover.de},\\
  \href{mailto:wenske@ifam.uni-hannover.de}{wenske@ifam.uni-hannover.de},
  \href{mailto:thomas.wick@ifam.uni-hannover.de}{thomas.wick@ifam.uni-hannover.de}}

\date{}

\begin{document}

\maketitle

\begin{abstract}
  This work develops and analyzes a
  variational-monolithic unfitted finite element formulation
  of a linear fluid-structure interaction problem
  in Eulerian coordinates with a fixed interface.
  The overall discretization is based on
  a backward Euler scheme in time
  and finite elements in space.
  For the spatial discretization we employ a cut finite element method
  on a mesh consisting of quadrilateral elements.
  We use a first-order in time formulation of the elasticity equations,
  inf-sup stable finite elements in the fluid part
  and Nitsche's method to incorporate the coupling conditions.
  Ghost penalty terms guarantee the robustness of the approach
  independently of the way the interface cuts the finite element mesh.
  The main objective is to establish stability and a priori error estimates.
  We prove optimal-order error estimates in space and time
  and substantiate them with numerical tests.
\end{abstract}



\section{Introduction}
This work is devoted to a numerical analysis of a
cut finite element approach for fully Eulerian fluid-structure interaction (FSI). Different methodologies to solve FSI problems have been published over the last decades,
such as the arbitrary Lagrangian Eulerian method~\cite{DoneaSurvey, Hughes1981, FoNo99}, fictitious
domain methods~\cite{Glowinski1994283,Glowinski2001363}, immersed boundary methods~\cite{Pe02},
and fully Eulerian methods~\cite{Du06,CoMaMi08}. Specializations
and further refinements of these techniques exist. Classical monographs
and textbooks on fluid-structure interaction
include
~\cite{FoQuaVe09,richter2017book,FrHoRiWiYa17,BaTaTe13}.
In fully Eulerian modeling,
the two subproblems are both described in Eulerian coordinates,
which specifically requires the rewriting of Lagrangian solids.
To set up a common system, a variational-monolithic coupling is employed.
First studies date back to~\cite{Du06,CoMaMi08}
and have since then been investigated and improved in other works,
such as \cite{richterWick2010,SuIiTaTaMa11,
Wi12_fsi_euler,RICHTER2013227,Sun20141,
RATH2023112188,FrKnStWeWi24_ENUMATH,FreiPhD, BurmanFernandezFrei2020}. We also mention our own recent computational work \cite{FrKnStWeWi25}, which lays out the computational framework for this work.
Here, the combination of the fully Eulerian approach for FSI with an unfitted finite element discretization was employed for the first time, yet without any error analysis. A major difference to existing studies in a mixed-coordinate framework~\cite{BURMAN2014FSI, BurmanFernandezGerosa2023} is the cut finite element formulation for the solid parts in (moving) Eulerian coordinates. This includes, for example, ghost penalty terms for the elasticity equations in the solid part, which have to date not been used in the literature within the context of FSI.

The advantage of a fully Eulerian formulation for FSI
lies in its ability to handle very large deformations,
topology changes, and contact problems in a straightforward way,
see, e.g.,
\cite{FrRiWi16_JCP,FrRi17,HechtPironneau2016,BurmanFernandezFrei2020, BurmanFernandezFreiGerosa2022}.
A shortcoming is the necessity of capturing and resolving the interface,
as the spatial finite element mesh is typically fixed, and the interface can move freely.
A fitted finite element framework for fully Eulerian FSI has been constructed in~\cite{FreiRichter2014, FreiPhD, FreiPressure2019}, at the cost of highly anisotropic mesh cells in the interface region. An elegant alternative is given by the cut finite element method~\cite{HANSBO2002elliptic, hansbo2005Nitsche, Burman2015CutFEM}, which is related to other unfitted finite element techniques, e.g., the extended finite element method (XFEM, \cite{MoesDolbowBelytschko99}).
In CutFEM,
interface conditions are imposed by means of Nitsche's method
\cite{Nitsche1971}, see \cite{HANSBO2002elliptic, hansbo2005Nitsche, BURMAN2014FSI}.
Moreover, additional stabilization of cut cells via ghost penalty terms is proposed,
since the condition number of the system matrix
suffers from cells cut into vastly different sizes,
see also \cite{BURMAN2010GhostPenalty, BURMAN2012FictitiousDomainII, burman_hansbo_2014}.
Further unfitted finite element methods include the Finite Cell method~\cite{ParvizianDuesterRank2007}, the Shifted Boundary Method~\cite{MainScovazzi2018} and $\phi$-FEM~\cite{DuprezLozinski2020}.

In the last decade, cut finite elements have been applied extensively to solve many different equations. Among them, we cite~\cite{MassingLarsonLoggRognes2014, HansboLarsonZahedi2014a, GuzmanOlshanskii2018, MassingSchottWall2017, Badia2018, HOANG2017400}
exemplarily for the Stokes equations and~\cite{StickoKreiss2019} for solving the wave equation. A dG cut method has been introduced in~\cite{BastianEngwer2009}; space-time cut finite elements have been developed in~\cite{LehrenfeldReusken2013, BadiaDilipVerdugo2023, HeimannLehrenfeldPreuss2023, HansboLarsonZahedi2016, AnselmannBause}. For a comprehensive overview, we refer to the recent Acta Numerica article~\cite{BurmanHansboLarsonZahedi2025}. FSI problems have been tackled previously in, e.g., \cite{BURMAN2014FSI, alauzet-et-al-15, kamensky-et-al-15, MassingLarsonetal2015, Liu2021,  BurmanFernandezGerosa2023}. However, in contrast to the present work, cut elements are used in these works only for the fluid part, while the solid equations are solved in Lagrangian coordinates on a fitted mesh, which is
glued to the unfitted fluid mesh.
Moreover, the only available convergence analysis for FSI in a CutFEM framework
is---to our knowledge---given in~\cite{BURMAN2014FSI, BurmanFernandezGerosa2023}.

The main objective (and thus the key difference to~\cite{FrKnStWeWi25}) of the current work is a rigorous numerical analysis of a CutFEM ghost penalty monolithic fully Eulerian fluid-structure interaction problem. In this work, we concentrate on a fixed interface and two spatial dimensions. In principle an extension of the analysis to 3d is possible, if the inf-sup stability
of the Taylor-Hood elements (on a fitted mesh) could be proven. To our knowledge,
however, this is an open problem on hexahedral meshes in 3d (see~\cite{Zulehner2024}). Moreover, the notation would become even more technical, such that we have decided to keep the restriction to two spatial dimensions.
A convergence analysis for Eulerian FSI on moving interfaces is out of scope of this article and has to our knowledge not yet been established in literature. First steps for parabolic and Stokes equations on moving domains are given in~\cite{Lehrenfeld2018, FreiSingh2024, vonWahlLehrenfeldRichter2022, burman2022eulerian, NeilanOlshanskii2024}. For the algorithmic implementation of these methods within the FSI context, we refer to~\cite{FrKnStWeWi25}.

As a first major step in our analysis we will derive
stability estimates. In contrast to previous works on CutFEM, we analyze a discretization on quadrilateral meshes using $Q_r$ finite elements.
Thus, the first step consists in the extension of known coercivity results for arbitrary cut cells using ghost penalty terms to $Q_r$ finite elements.

Moreover, in contrast to~\cite{BURMAN2014FSI, BurmanFernandezGerosa2023}, the estimates are complicated by the fact that we use inf-sup stable finite elements instead of equal-order finite elements with pressure stabilization. For this reason, the derivation of $L^2(H^1)$-stability estimates for the pressure will be necessary. Here, we follow~\cite{GuzmanOlshanskii2018, NeilanOlshanskii2024} and assume the inverse parabolic CFL condition, $\smash[b]{h^2 \lesssim k}$, where $h$ is the spatial discretization parameter and $k$ the temporal discretization term.

Having the stability estimates at hand, we establish consistency and interpolation estimates and prove the final optimal-order a priori error estimates in an energy norm including an optimal $L^2(H^1)$-estimate of the pressure error.
 Afterwards, we present numerical results to investigate the optimality of the proven error estimates under mesh refinement. The computational descriptions are kept short as they have been discussed extensively in our prior work~\cite{FrKnStWeWi25}.

The outline of this paper is as follows. In \cref{sec_notation_model},
the notation and the strong and weak formulations are introduced.
Next, in \cref{sec_disc}, the discretization using an unfitted
finite element method with ghost penalty stabilization is provided.
In the main \cref{sec_num_ana,sec_apriori},
stability and a priori error estimates are established.
Lastly, in \cref{sec_tests},
numerical experiments are conducted in order to substantiate our theory.

\section{Notation and model}
\label{sec_notation_model}

Let $\Omega \subset \R^2$ be a bounded domain,
partitioned into fluid and solid domains
$\Omega_f$ and $\Omega_s$ with Lipschitz boundaries
such that $\cl\Omega = \cl\Omega_f \cup \cl\Omega_s$
with $\Omega_f \cap \Omega_s = \0$.
Denote by $\Gamma^i := \cl\Omega_f \cap \cl\Omega_s$
the interface between fluid and solid
and by $\Gamma_f^D := \bd\Omega_f \cap \bd\Omega\, \neq \emptyset$
and $\Gamma_s^D \subset \bd\Omega_s \cap \bd\Omega$
Dirichlet boundaries for the fluid and solid subdomains,
respectively (see \cref{fig:domains}).
Let $\intI = (0, T)$ be a given time interval.
To simplify the numerical analysis,
assume that the domain partitioning is constant (fixed) in time. {{A}ccordingly, we do not consider convective terms} in
the Eulerian formulation of the solid equations.
The strong formulation of the coupled system then reads as follows.

\begin{problem}[Strong formulation]%
  \label{strong_form}%
  \newlength\lone\newlength\ltwo\settowidth\lone{$\rho_f f_f$}%
  \settowidth\ltwo{$\textup{on }
    (\bd\Omega_s \setminus (\Gamma_s^D \cup \Gamma^i)) \times \intI$,}%
  \newcommand\rhs[3]{\hbox to\lone{$#1$\hfil}\,
    \hbox to\ltwo{\textup{#2} $#3$\hfil}}%
  \newcommand\br{\\[-3pt]}%
Consider given force fields
$f_f\: \cl\Omega_f \times \clI \to \R^2$ and
$f_s\: \cl\Omega_s \times \clI \to \R^2$, %
the fluid velocity $v_f\: \cl\Omega_f \times \clI \to \R^2$,
the pressure $p\: \cl\Omega_f \times \clI \to \R$,
the solid velocity $v_s\: \cl\Omega_s \times \clI \to \R^2$
and the displacement $u\: \cl\Omega_s \times \clI \to \R^2$
with initial values
$\smash[b]{v_f^0}\: \Omega_f \to \R^2$,
$v_s^0\: \Omega_s \to \R^2$ and
$u^0\: \Omega_s \to \R^2$.
Then, find $U := (v_f, p, v_s, u)$,
such that
  \begin{equation*}
    \begin{alignedat}{3}
      \rho_f \partial_t v_f - \div \sigma_f &= \rho_f f_f, \,&\
      \div v_f &= 0 \textup{\kern7pt in } \Omega_f \times \intI, &
      v_f &= 0 \textup{ on } \Gamma_f^D \times \intI,
      \\
      \rho_s \partial_t v_s - \div \sigma_s &= \rho_s f_s, &
      \partial_t u &= v_s \textup{ in } \Omega_s \times \intI, &
      u &= 0 \textup{ on } \Gamma_s^D \times \intI,
      \\
      \sigma_s \n &= 0 \textup{\kern6pt on }
      (\bd\Omega_s \setminus (\Gamma_s^D\cup \Gamma^i)) \times \intI,
      \mkern-200mu && &
      v_f &= v_s, \ \sigma_f \n = \sigma_s \n
      \textup{ on } \Gamma^i \times \intI,
      \\
      u &= u^0, &
      v_s &= v_s^0 \textup{ in } \Omega_s \times \set{0}, &\quad
      v_f &= v_f^0 \textup{ in } \Omega_f \times \set{0}.
    \end{alignedat}
  \end{equation*}
Let $\symgrad w := \frac12 (\nabla w + \nabla w^\top)$
denote the symmetrized gradient of a function $w$.
Define the fluid stress as
$\sigma_f \equiv \sigma_f(v_f, p) := 2 \rho_f \nu_f \symgrad{v_f} - p \mathbb I$
and, using a linearized Saint Venant-Kirchhoff model,
the solid stress by
$\sigma_s \equiv \sigma_s(u) :=
2 \mu_s \symgrad{u} + \lambda_s \tr{\symgrad{u}} \mathbb I$
with the Lamé parameters $\mu_s$ and $\lambda_s$,
the densities $\rho_f$ and $\rho_s$ (constant in space and time),
the fluid viscosity $\nu_f$,
the surface normal $\n$,
and the identity matrix $\mathbb I \in \R^{2 \times 2}$.
\end{problem}
For the weak formulation, define the spaces
$\V_f := H^1_0(\Omega_f; \Gamma_f^D)^{{2}}$, $\V_s := {{H}^1_0(\Omega_s; \Gamma_s^D)}^{{2}}$,
$\U := H^1_0(\Omega_s; \Gamma_s^D)^{{2}}$, $\P := L^2(\Omega_f)$
and let $\X := \V_f \times \P \times \V_s \times \U$ be their product space.
For function spaces $Y$ and $Z$
over $\Omega_i$ for $i \in \set{f, s}$, such that
$Y \subset {L^2(\Omega_{i})}^{{2}} \subset Y^*$ form a Gelfand tripel
(with $Y^*$ denoting the topological dual space of $Y$), let
$
W(\intI, Y, Z) :=
\defset{v \in L^2(\intI, Y)}{\partial_t v \in L^2(\intI, Z)}
$
with the Bochner spaces $L^2(\intI, \fcdot)$.
We combine the temporal and spatial function spaces into the space
$$W(\intI, \X) :=
W(\intI,\V_f,\V_f^*) \times L^2(\intI,\P) \times
W(\intI, \V_s, \V_s^*) \times
W(\intI, \U, {L^2(\Omega_s)}).$$
This choice of function spaces yields a continuous embedding
of the velocities $v_f, v_s$ and the deformation~$u$ into spaces
$C^0(\clI, L^2(\Omega_i))$ \cite{Dautray1992Vol5}.
Hence, the prescription of initial values $v_f^0, v_s^0, u^0$ is well-defined.
In the following, we state the weak formulation in a slightly more general form,
which will be convenient for the numerical analysis.
To this end, we introduce a functional $F \in \X^*$
for the right-hand side of the weak formulation, defined as
\begin{align*}
  F(\phi_f, \xi, \phi_s, 0) := F_f(\phi_f) + F_p(\xi) + F_s(\phi_s).
\end{align*}
Hereby, \cref{strong_form} is covered as the case where
$F_f(\phi_f) = \rho_f \prodOf{f_f}{\phi_f}$, $F_p(\xi) = 0$ and
$F_s(\phi_s) = \rho_s \prodOs{f_s}{\phi_s}$.
Multiplication with test functions and integration by parts
then translates \cref{strong_form} into the following weak formulation, where $\dpair[Y^* \times Y]{\fcdot}{\fcdot}$
denotes the duality pairing of $Y^*$ and $Y$.

\begin{problem}[Weak formulation]%
  \label{weak_form}%
  Find $(v_f, p, v_s, u)$ $\in$
  $W(\intI,\X)$
  for given initial values $v_f^0,\,v_s^0,\,u^0$
  such that $v_f = v_s$ almost everywhere on $\Gamma^i$,
  $\partial_t u = v_s$ almost everywhere in $\Omega_s$, and for all
  $(\phi_f, \xi, \phi_s, 0)$ in $\X$
  and almost all $t \in \intI$:
  \begin{multline*}
    \rho_f \pairf{\partial_t v_f}{\phi_f} +
    \prodOf{\sigma_f}{\nabla\phi_f} + \prodOf{\div v_f}{\xi}
    + \rho_s \pairs{\partial_t v_s}{\phi_s} \\[-2pt]
    + \prodOs{\sigma_s}{\nabla\phi_s}
      - \pairi{\sigma_f \n_f}{\phi_f - \phi_s}
      = F(\phi_f,{{\xi},\phi_s},0).
  \end{multline*}
\end{problem}

\section{Spatial and temporal discretization}
\label{sec_disc}

\begin{figure}[t]
\begin{minipage}[c]{.3\linewidth}
\centering
 \begin{tikzpicture}[font = \scriptsize, scale = 1.5]
  \draw[color=blue, fill=blue!10, thick] (-1, -1) rectangle (1, 1);
  \draw[color=red, fill=red!10, thick] (0, 0) circle (0.75);
  \draw (-0.7, 0.7) node {\scriptsize \textcolor{blue}{$\Omega_f$}};
  \draw (0, 0) node {\scriptsize \textcolor{red}{$\Omega_s$}};
  \draw[anchor=south] (0, -0.7) node {\scriptsize \textcolor{red}{$\Gamma^i=\bd\Omega_s$}};
  \draw[anchor=south] (0, -1.07) node {\scriptsize \textcolor{blue}{$\bd\Omega=\bd\Omega_f\setminus\Gamma^i$}};
 \end{tikzpicture}
\end{minipage}
\begin{minipage}[c]{.3\linewidth}
\centering
 \begin{tikzpicture}[font = \scriptsize, scale = 1.5]
    \draw[color=blue!10, fill=blue!10] (-1, -1) rectangle (1, 1);
\ifXcolor   \fill[pattern=north east lines, pattern color=Fuchsia] (-0.6, -0.8) rectangle (0.6, 0.8);
    \fill[pattern=north east lines, pattern color=Fuchsia] (-0.8, -0.6) rectangle (0.8, 0.6);
    \draw[color=red!10, fill=red!10] (-0.4, -0.6) rectangle (0.4, 0.6);
    \draw[color=red!10, fill=red!10] (-0.6, -0.4) rectangle (0.6, 0.4);
    \draw[step=0.2,black!60,thin] (-1, -1) grid (1, 1);
    \draw[color=black, thick] (-1, -1) rectangle (1, 1);
    \draw[thick] (0, 0) circle (0.75);
    \draw[color=blue!10, fill=blue!10] (-0.8, 0.8) circle (0.13);
    \draw (-0.8, 0.8) node {\scriptsize \textcolor{blue}{$\Oh_f$}};
    \draw[color=blue!10, fill=blue!10] (0.8, -0.8) circle (0.13);
\ifXcolor    \draw (0.8, -0.8) node {\scriptsize \textcolor{Fuchsia}{$G^h$}};
    \draw[color=red!10, fill=red!10] (0, -0.4) circle (0.13);
    \draw (0, -0.4) node {\scriptsize \textcolor{PineGreen}{$\F_G^f$}};
    \draw[color=red!10, fill=red!10] (0, 0) circle (0.16);
    \draw (0, 0) node {\scriptsize \textcolor{red}{$\isolid$}};

    \draw[color=PineGreen, very thick] (-0.6, -0.8) -- (0.6, -0.8);
    \draw[color=PineGreen, very thick] (-0.6, 0.8) -- (0.6, 0.8);
    \draw[color=PineGreen, very thick] (-0.8, -0.6) -- (-0.8, 0.6);
    \draw[color=PineGreen, very thick] (0.8, -0.6) -- (0.8, 0.6);
    \foreach \x in {-0.6,-0.4,-0.2,...,0.6}
    {
        \draw[color=PineGreen, very thick] (\x, 0.6) -- (\x, 0.8);
        \draw[color=PineGreen, very thick] (\x, -0.6) -- (\x, -0.8);
        \draw[color=PineGreen, very thick] (0.6, \x) -- (0.8, \x);
        \draw[color=PineGreen, very thick] (-0.6, \x) -- (-0.8, \x);
    }
    \foreach \x in {-1,1}
        \foreach \y in {-1,1}
        {
            \draw[color=PineGreen, very thick] (\x*0.6, \y*0.6) -- (\x*0.4, \y*0.6);
            \draw[color=PineGreen, very thick] (\x*0.6, \y*0.6) -- (\x*0.6, \y*0.4);
        }\fi\fi
 \end{tikzpicture}
\end{minipage}
\begin{minipage}[c]{.3\linewidth}
\centering
 \begin{tikzpicture}[font = \scriptsize, scale = 1.5]
    \draw[color=black, fill=blue!10, thick] (-1, -1) rectangle (1, 1);
    \draw[color=red!10, fill=red!10] (-0.6, -0.8) rectangle (0.6, 0.8);
    \draw[color=red!10, fill=red!10] (-0.8, -0.6) rectangle (0.8, 0.6);
\ifXcolor   \fill[pattern=north east lines, pattern color=Fuchsia] (-0.6, -0.8) rectangle (0.6, 0.8);
    \fill[pattern=north east lines, pattern color=Fuchsia] (-0.8, -0.6) rectangle (0.8, 0.6);
    \draw[color=red!10, fill=red!10] (-0.4, -0.6) rectangle (0.4, 0.6);
    \draw[color=red!10, fill=red!10] (-0.6, -0.4) rectangle (0.6, 0.4);
    \draw[step=0.2,black!60,thin] (-1, -1) grid (1, 1);
    \draw[thick] (0, 0) circle (0.75);
    \draw[color=red!10, fill=red!10] (0, 0) circle (0.13);
    \draw (0, 0) node {\scriptsize \textcolor{red}{$\Oh_s$}};
    \draw[color=blue!10, fill=blue!10] (0.8, -0.8) circle (0.13);
 \ifXcolor   \draw (0.8, -0.8) node {\scriptsize \textcolor{Fuchsia}{$G^h$}};
    \draw[color=red!10, fill=red!10] (0, -0.4) circle (0.13);
    \draw (0, -0.4) node {\scriptsize \textcolor{Bittersweet}{$\F_G^s$}};
    \draw[color=blue!10, fill=blue!10] (-0.8, 0.8) circle (0.16);
    \draw (-0.8, 0.8) node {\scriptsize \textcolor{blue}{$\ifluid$}};\fi

    \draw[color=Bittersweet, very thick] (-0.6, -0.6) -- (0.6, -0.6);
    \draw[color=Bittersweet, very thick] (-0.6, 0.6) -- (0.6, 0.6);
    \draw[color=Bittersweet, very thick] (-0.6, -0.6) -- (-0.6, 0.6);
    \draw[color=Bittersweet, very thick] (0.6, -0.6) -- (0.6, 0.6);
    \foreach \x in {-0.4,-0.2,...,0.4}
    {
        \draw[color=Bittersweet, very thick] (\x, 0.6) -- (\x, 0.8);
        \draw[color=Bittersweet, very thick] (\x, -0.6) -- (\x, -0.8);
        \draw[color=Bittersweet, very thick] (0.6, \x) -- (0.8, \x);
        \draw[color=Bittersweet, very thick] (-0.6, \x) -- (-0.8, \x);
    }
    \foreach \x in {-1,1}
        \foreach \y in {-1,1}
        {
            \draw[color=Bittersweet, very thick] (\x*0.4, \y*0.4) -- (\x*0.4, \y*0.6);
            \draw[color=Bittersweet, very thick] (\x*0.4, \y*0.4) -- (\x*0.6, \y*0.4);
        }\fi
 \end{tikzpicture}
\end{minipage}
\caption{Example of a domain
  partitioned into subdomains $\Omega_f$ and $\Omega_s$ (left)
  and a visualization of the corresponding computational domains
  with the sets of ghost penalty edges for the fluid ($\F_G^f$, middle)
  and the solid domain ($\F_G^s$, right).}
\label{fig:domains}
\end{figure}
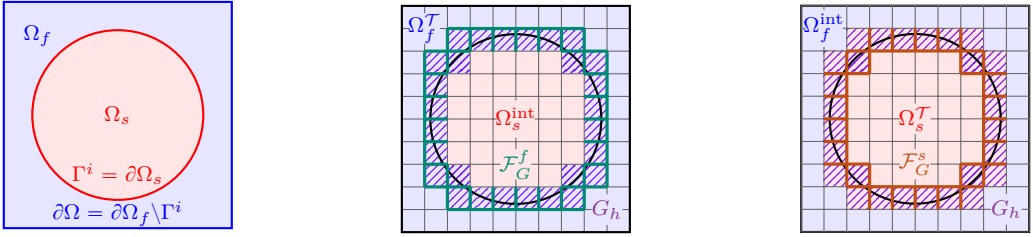

In this section, we introduce the numerical method
by describing spatial and temporal discretizations.
\paragraph{Spatial discretization}

For the spatial discretization of the domain $\Omega$
we consider a family of quasi-uniform triangulations
$\set{\T^h}$ into quadrilaterals
that are fitted to the boundary $\bd\Omega$,
but not to the interface $\Gamma^i$.
The mesh size parameter is defined by $h = \max_{K\in\T^h} h_K$,
where $h_K$ denotes the diameter of the cell $K$.
We define overlapping fluid and solid subtriangulations
$\T_f^h$ respectively $\T_s^h$ and denote the computational domains
spanned by them by $\Omega_f^\T$ respectively $\Omega_s^\T$:
\begin{equation*}
  \T_i^h := \defset{K \in \T^h}{K \cap \Omega_i \ne \0}, \quad
  \Omega_i^\T := \bigcup\nolimits_{K \in \T_i^h} \cl K, \quad
  i \in \set{f, s}.
\end{equation*}
Each computational domain $\Omega_i^\T$ contains its respective physical domain $\Omega_i$
and the portions of cells
cut by the interface that lie inside the other physical domain
(see \cref{fig:domains}).
Let $G^h = \Ofh \cap \Osh$ denote
the interface zone{, $\interf:=\Omega_i\setminus G^h$,
$i\in\{f,s\}$, the interior subdomains (see \cref{fig:domains})}
and $\F_h$ the set of all edges of~$\T^h$.

For the fluid, we use Taylor--Hood finite elements of order $m_f \ge 2$,
i.e., continuous $Q_{m_f}$ elements for the fluid velocity
and continuous $Q_{m_f-1}$ elements for the pressure.
For the solid variables $u$ and $v_s$,
we use equal-order $Q_{m_s}$ elements of order $m_s\ge 1$.

For each element $K \in \T_i^h$, there exists a bilinear bijective map
$\xi_K\: \hat{K} \to K$ from the unit square $\hat{K} = (0,1)^2$.
We introduce the finite element spaces
\begin{align}
V_i^{h, r} := \defset{\phi \in C(\cl{\Oh_i})}
{(\phi|_K \circ \xi_K) \in Q_r(\hat{K})\, \, \forall K \in \T_i^h},
\quad i \in \set{f, s},\label{FEspaces}
\end{align}
and set $V_i^D := H_0^1(\Omega_i; \Gamma_i^D)$ to define
\begin{align*}
  \V_f^h &:= (V_f^{h, m_f} \cap V_f^D)^2,
  &\P^h &:= V_f^{h, m_f-1}, \\
  \V_s^h &:= (V_s^{h, m_s}{{} \cap V_s^D})^2,
  &\U^h &:= (V_s^{h, m_s} \cap V_s^D)^2.
\end{align*}
The product space is denoted by
$\X^h := \V_f^h \times \P^h \times \V_s^h \times \U^h$.

To extend the stability of the bilinear form to the computational domains
$\Ofh$, $\Osh$,  we use ghost penalty terms.
In this work, we focus on the classical ghost penalty method
based on penalizing jumps over derivatives.
Other stabilization techniques are possible, see e.g.,
\cite{Lehrenfeld2018,BURMAN2010GhostPenalty}.
Let $\F_G^f\subset \F_h$ denote the set of element edges
$F = \cl K_1 \cap \cl K_2$ of the triangulation
$K_1, K_2\in \T_f^h$
that do not lie on the boundary $\bd\Omega$,
such that at least one of the cells $K_j$ is intersected by the interface:
$K_j \cap \Gamma^i \ne \0$ for $j \in \set{1, 2}$.
Analogously, we define $\F_G^s$
as the set of corresponding edges of the triangulation $\T_s^h$
(see \cref{fig:domains}).
For a cell $K$ cut by the interface we denote by $K_f$ and $K_s$
the part of the cell inside the respective subdomain.
For neighboring cells $K,\,K'$ with common edge $F = \cl K \cap \cl K'$,
let
$\jump{\partial_{\n}^i \varphi} :=
\partial_{\n}^i \varphi|_K - \partial_{\n}^i \varphi|_{K'}$
denote the jump in the $i$-th partial derivatives of a function~$\varphi$.
Using
$\Jump{i} := \iprod[F]
{\jump{\partial_{\n}^i \varphi_1}}
{\jump{\partial_{\n}^i \varphi_2}}$,
the ghost penalty functions (with positive parameters
$\gamma_{v_f}$, $\gamma_p$, $\gamma_{v_s}$, $\gamma_u$)
are defined as follows:
\newcommand\sumFf{\sum_{F \in \F_G^f}}
\newcommand\sumFs{\sum_{F \in \F_G^s}}
\ifcase0%
\newcommand\sumFKf{\sumFf w_F^f}
\newcommand\sumFKs{\sumFs w_F^s}
\or
\newcommand\sumFKf{\sumFf \sum_{K\mathpunct: F \in \cl K} w(\kappa_f)}
\newcommand\sumFKs{\sumFs \sum_{K\mathpunct: F \in \cl K} w(\kappa_s)}
\fi
\begin{align}\begin{split}\label{eq:ghost}
	\ghw[v_f]{\varphi_1}{\varphi_2}
	&:= \gamma_{v_f} \sumFKf
	\sum_{l=1}^{m_f} \frac{h^{2l-1}}{(l-1)!^2} \Jump{l},\\
	\ghw[p]{\varphi_1}{\varphi_2}
	&:= \gamma_p \sumFKf
	\sum_{l=1}^{m_f-1} \frac{h^{2l+1}}{(l!)^2} \Jump{l}, \\
	\ghw[v_s]{\varphi_1}{\varphi_2}
	&:= \gamma_{v_s} \sumFKs \sum_{l=1}^{m_s} \frac{h^{2l+1}}{(l!)^2} \Jump{l}, \\
	\ghw[u]{\varphi_1}{\varphi_2}
	&:= \gamma_u \sumFKs \sum_{l=1}^{m_s} \frac{h^{2l-1}}{(l-1)!^2} \Jump{l}.
	 \end{split}
\end{align}
Here, we set $w_F^i := \sum_{K\mathpunct: F \subset \cl K} w(\kappa_i)$
for $i \in \set{f, s}$, using an additional weight function
$w\: [0,1] \to [\frac12 \wmax^{-1},\frac12 \wmax]$,
$\kappa \mapsto \frac12 \wmax^{1 - 2 \kappa}$, with $\wmax \ge 1$,
which scales the ghost penalties depending on the fraction of the cell
inside the respective domain $\Omega_i$,
$i \in \set{f, s}$, denoted as $\kappa_i := \text{meas}(K_i) / \text{meas}(K)$.
Thus we penalize ``bad cuts'' more severely
while ``good cuts'' (where a sufficiently large fraction of the cell lies inside)
are penalized less severely.
Moreover,
the conventional ghost penalty terms are recovered as the special case
where $\wmax = 1$, hence $w \equiv 0.5$ and $w_F^i = 1$.

  For the FSI coupling, we make use of Nitsche's method,
  see, e.g., \cite{BURMAN2014FSI}.
  Let $\Phi^h := (\phi_f^h, \xi^h, \phi_s^h, \psi^h)$ to
  introduce the following notation for the interface terms,
  the fluid and solid bulk terms as well
  as the ghost penalty stabilization:
  \begin{align*}
    j^h(U^h,\Phi^h)
    &:= h^{-1} \rho_f\nu_f\gamma_N \prodGi{v_f^h-v_s^h}{\phi_f^h-\phi_s^h} \\[-2pt]
    &\qquad- \prodGi{\sigma_f^h(v_f^h,p^h)\n_f}{\phi_f^h-\phi_s^h}
    - \prodGi{v_f^h-v_s^h}{\sigma_f^h(\phi_f^h,-\xi^h)\n_f}, \mkern-200mu \\
    a_f^h(U^h,\Phi^h)
    &:= %
    \prodOf{\sigma_f^h}{\nabla\phi_f^h} + \prodOf{\div v_f^h}{\xi^h},
    &a_s^h(U^h,\Phi^h)
    &:= %
     \prodOs{\sigma_s^h}{\nabla\phi_s^h}, \\
    S_f^h(U^h,\Phi^h)
    &:= 2\rho_f\nu_f \ghw[v_f]{v_f^h}{\phi_f^h} + \ghw[p]{p^h}{\xi^h},
    &S_s^h(U^h,\Phi^h)
    &:= %
      2\mu_s \ghw[u]{u^h}{\phi_s^h}.
  \end{align*}
  Combine the terms to obtain the following
  bilinear forms and right-hand side function:
  \begin{align*}
    M^h(U^h, \Phi^h)
    &:= \rho_f \prodOf{v_f^h}{\phi_f^h}
     + \rho_s\prodOs{v_s^h}{\phi_s^h}
     + \rho_s \ghw[v_s]{v_s^h}{\phi_s^h}, \\
    A^h(U^h, \Phi^h)
    &:= (a_f^h + a_s^h + j^h)(U^h, \Phi^h), \quad
    S^h(U^h, \Phi^h)
    := (S_f^h + S_s^h)(U^h, \Phi^h), \\
   F^h(\Phi^h) &:= F_f(\phi_f^h) + F_p(\xi^h) + F_s(\phi_s^h).
  \end{align*}

\paragraph{Temporal discretization}

For the discretization in time, we apply the backward Euler method
using (for simplicity)
a uniform grid $\set{t_n}_{n=0}^N$ with $t_n := n k$, $k := T / N$.
We define the right-hand side $F^n(\Phi^h) := F^h(t_n)(\Phi^h)$,
where $F^h(\fcdot)$ depends implicitly on time
and $F^h(t_n)(\fcdot)$ restricts it to time $t_n$.
We use $\delta U^{h,n} := U^{h,n} - U^{h,n-1}$
(and similarly $\delta u^{h,n}$ etc.)\
to denote differences between two time steps. With residuals
\begin{equation}\label{defvres}
\vres := v_s^{h,n} - k^{-1} \delta u^{h,n},
\end{equation} the fully discretized problem is then given as follows.

\begin{problem}[Fully discrete formulation]%
  \label{prob:fullydisc}%
  Given initial values $(v_f^{0}, v_s^{0}, u^0)$,
  find for $n = 1, \dots, N$
states $U^{h,n} := (v_f^{h,n}, p^{h,n}, v_s^{h,n}, u^{h,n})$
  in
  $\X^h$, such that
  $\vres = 0$ and for all
  $\Phi^h = (\phi_f^h, \xi^h, \phi_s^h, \psi^h)$ in $\X^h$:
  \begin{align}
  \begin{split}\label{eq:fullydisc}
    M^h(\Udiff, \Phi^h)
    + k A^h(U^{h,n},\Phi^h) + k S^h(U^{h,n},\Phi^h)
    &= kF^n(\Phi^h).
    \end{split}
  \end{align}
  {We set
  $v_f^{h,0} := I_h E_f v_f^0$, $u^{h,0} := I_h E_s u^0$
  and $v_s^{h,0} := I_h E_s v_s^0$.
  Here, $E_i$ denotes a suitable extension operator (see \cref{sec_apriori})
  and $I_h$ denotes the nodal interpolant
  into the respective finite element space.}
  By defining
\begin{align}
  \label{eq:definition:Akh}
    \A^{kh}(U^{h,n}, U^{h,n-1}; \Phi^h)
    &:=M^h(\Udiff, \Phi^h) + k A^h(U^{h,n}, \Phi^h) + k S^h(U^{h,n}, \Phi^h),
  \end{align}
  we can abbreviate \eqref{eq:fullydisc} as
  $\A^{kh}(U^{h,n}, U^{h,n-1}; \Phi^h) = k F^n(\Phi^h)$.
\end{problem}

\begin{remark}
  Even though $\psi^h$ (the fourth component of the test function $\Phi^h$)
  does not appear in the fully discrete formulation,
  we keep it in the definition,
  as it will simplify our notation in the following.
\end{remark}
\begin{remark}
\label{rem_3_3}
Finally, we note that the relation $\vres =v_s^{h,n} - k^{-1} \delta u^{h,n} =0$
is included in strong form in Problem~\ref{prob:fullydisc}. This will shorten the presentation in the following. In our practical implementation, this condition is
enforced weakly as an $L^2$-projection (see~\eqref{eq_uhn_disc} in Section~\ref{sec_tests}).
Testing with $\psi^h = \delta u^{h,n} - k v_s^{h,n}$, one can show that this is in fact equivalent to a pointwise enforcement of
$\vres = 0$. We also refer to the presentation in~\cite{BURMAN2014FSI},
where also the strong form is used.
\end{remark}

\section{Stability analysis}
\label{sec_num_ana}

In this section, we derive a stability estimate for the fully discrete scheme,
which will be the basis for the a priori error estimate in \cref{sec_apriori}.
Our derivations are divided into three subsections.
First, we show that the ghost penalty terms
serve to extend stability from the physical to the computational domains. While this is well-known for triangular meshes, we prove here a generalization of this result to quadrilateral meshes.
Then, a stability estimate in an energy norm is shown.
This is followed by a specific stability estimate for the pressure.
Throughout the analysis, we will assume that the interface $\Gamma^i$
as well as the subdomains $\Omega_f$ and $\Omega_s$ are sufficiently smooth,
and that the discretization parameters $h$ and $k$ are sufficiently small.
We start by showing that the ghost penalty terms defined in~\eqref{eq:ghost}
serve to extend stability from the physical to the computational domains.
To this purpose, we need the following assumption
(see also~\cite{Lehrenfeld2018}).

\begin{assumption}%
  \label{FKK}%
  For every $K \subseteq G^h$ and $i \in \set{f, s}$ there exists
  $K_i \subseteq \Omega_i \setminus G^h$ %
  that is reached from $K$
  by crossing a finite number $M \le N_\F$ of edges
  in $\smash[b]{\Oih}$, %
  i.e., there exists a path $(K_0, \dots, K_M)$
  with $K_l\subset \Omega_i^{{\cal T}}$ and $\cl{K}_l \cap \cl{K}_{l+1} \in \F_G^i$ for $l=0,\ldots,M-1$,
  $K_0 = K$ and $K_M = K_i$. %
  Denoting by $N_K$ how often $K \in \T^h$
  is used as final element
  $K_i$ among all $K' \subseteq G^h$,
  we assume further that all frequencies $N_K$
and the maximum number of edges $N _\F > 0$
  are bounded independently of the mesh size $h$.
  {Moreover, we assume that every element $K\in\T_f^h$ with
  $K\subset\ifluid$ has at least one vertex that does not lie on
  $\partial\ifluid$.}
\end{assumption}
The path condition in Assumption~\ref{FKK} is reasonable if $\Gamma^i$ is smooth
and the mesh $\T^h$ is sufficiently fine.
For a detailed discussion, we refer to~\cite[Rem.~5.2]{Lehrenfeld2018}.
We will also use the following trace and inverse inequalities, see e.g.,
\cite[Lem.~2, Lem.~3]{BURMAN2012FictitiousDomainII}. %
\begin{lemma}%
  \label{traceineq}%
  Let $K \in \T^h$ and $v \in H^1(K)$.
  Then there exist $C_T, C\tsb{TI} > 0$ such~that
  $$
  \norm[\bd{K}]{v} \le
  C_T \bigl( h^{-\frac12} \norm[K]{v} + h^{\frac12} \norm[K]{\nabla v} \bigr),
  \quad
  \norm[K \cap \Gamma^i]{v} \le
  C\tsb{TI} \bigl( h^{-\frac12} \norm[K]{v} + h^{\frac12} \norm[K]{\nabla v} \bigr)
  .
  $$
  For a finite element function $v_h \in {\cal V}_i^{h,r}$,
    $i \in \set{f,s}$, it holds with a constant $C\tsb{inv}$
  $$
  \norm[K \cap \Gamma^i]{v_h} \le
  C\tsb{inv} \bigl( h^{-\frac12} \norm[K]{v_h} \bigr).
  $$
\end{lemma}
Note that these estimates are valid for any smooth interface $\Gamma^i$.

\subsection{Extension of stability to the computational domains}

To prove stability of the semi-discrete formulation,
norms of certain functions on the computational domains
need to be bounded by their respective norms on the physical domains.
For $P_r$ finite elements on triangular meshes ($r\ge 1$),
the following estimates are well-established,
see~\cite{BURMAN2012FictitiousDomainII} for $r=1$
and~\cite{MassingLarsonLoggRognes2014} for $r>1$.
To our knowledge, the following proof
for $Q_r$ elements on quadrilateral meshes is new.

\begin{lemma}%
  \label{lem_stab_gen}%
  Let Assumption~\ref{FKK} hold, let $r \in \N$ and $i \in \set{f, s}$.
  Any $v^h \in V_i^{h, r}$
  then satisfies the following estimate
  with some $C > 0$ for $l \in \set{0, 1}$,
  where $\nabla^0 := \textup{id}$:
  \begin{align}\label{ghost_gen}
   C \norm[\Oih]{\nabla^l v^h}^2
   &\le \norm[\interf]{\nabla^l v^h}^2
     + \sum_{j=1}^r \frac{h^{2(j-l)+1}}{(j-l)!^2}
     \ifcase0 \sumFi w_F^i \or \sumFKi w(\kappa_i) \fi g_F^j(v^h,v^h)
     .
 \end{align}
\end{lemma}

\begin{remark}
  \Cref{lem_stab_gen} implies in particular the following estimates
  for functions $v_s^h, u^h \in V_s^{h, m_s}$, $p^h \in V_f^{h, m_f-1}$
  and $v_f^h \in V_f^{h, m_f}$ with constants $C_{v_f}, C_p, C_{v_s}, C_u > 0$:
  \begin{align}
    \label{est:vf}
    C_{v_f} \normOfh{\nabla v_f^h}^2
    &\le \normIf{\nabla v_f^h}^2
      + \ghw[v_f]{v_f^h}{v_f^h}, \\
      \label{est:p}
    C_p \normOfh{p^h}^2
    &\le \normIf{p^h}^2 + \ghw[p]{p^h}{p^h}, \\
    \label{est:vs}
    C_{v_s} \normOsh{v_s^h}^2
    &\le \normIs{v_s^h}^2 + \ghw[v_s]{v_s^h}{v_s^h}, \\
    \label{est:u}
    C_u \normOsh{\nabla u^h}^2
    &\le \normIs{\nabla u^h}^2 + \ghw[u]{u^h}{u^h}.
  \end{align}
\end{remark}

\begin{proof}[Proof (\cref{lem_stab_gen})]
  We start by splitting
  $\smash[b]{\norm[\Oih]{\nabla^l v^h}^2}
  = \smash[b]{\norm[\interf]{\nabla^l v^h}^2}
  + \norm[G^h]{\nabla^l v^h}^2$.
In order to estimate
$\norm[G^h]{\nabla^l v^h}$ in terms of
$\smash[b]{\norm[\interf]{\nabla^l v^h}}$, let $\hat{K}$ be the reference unit square
  and for each $K \in \T_i^h$ let $\xi_K\: \hat{K} \to K$
  be
  the (unique) bijective regular map $\xi_K\in Q_1$. Let $v := v^h|_K$.
  By definition of $V_i^{h, r}$ (see \eqref{FEspaces}),
  it holds $(v \circ \xi_K) \in Q_r(\hat K)$.
  Let $K' \in \T_i^h$ be a neighboring element of $K$
  with the common edge $F := \cl{K} \cap \cl{K'}$.
  Let $\hat v(\hat x) = v(\xi_K(\hat x))$
  and let $\hat{F}:=\xi_K^{-1}(F)$ be,
  without loss of generality, a vertical edge of~$\hat{K}$
  (see \cref{fig:lem_stab_gen}).

  As the map $\xi_K$ may be nonlinear,
  the unit normal vector $\hat{\n}$ of $\hat{F}$
  is not necessarily mapped to a normal vector $\n$
  of the corresponding edge $F$ of the element $K$ by $\xi_K$,
  i.e., $(\hat\nabla \xi_K) \hat\n$ and
  $\n = \frac{(\hat\nabla \xi_K)^{-\top} \hat\n}{\|(\hat\nabla \xi_K)^{-\top} \hat\n\|}$
  might point in different directions (see \cref{fig:lem_stab_gen}).
  We denote the normalized vector that $\xi_K$ maps to $\n$
  by $\hat\n_\xi(\hat{x})=\frac{(\hat\nabla \xi_K)^{-1} \n}{\|(\hat\nabla \xi_K)^{-1} \n\|}$. By the chain rule, it holds that $\hat\nabla \hat{v}(\hat{x})=(\hat\nabla \xi_K)^{\top} \nabla v(x)$ for $x = \xi_K(\hat{x})$, and, as a consequence
  \begin{align}
    \label{normalDer}
    \hat{\partial}_{\n,\xi} \hat{v}(\hat{x})
    :=
    \hat\nabla \hat{v}(\hat{x}) \hat\n_{\xi}
    =(\hat\nabla \xi_K)^{\top} \nabla v(x) \cdot \frac{(\hat\nabla\xi_K)^{-1} \n}{\|(\hat\nabla \xi_K)^{-1} \n\|} =
    \frac{\partial_{\n} v(x)}{\|(\hat\nabla \xi_K)^{-1} \n\|}.
  \end{align}
  We will first show the estimate
  \begin{align}
    \label{hatF2}
    \norm[\hat K]{\hat\nabla^l\hat v}^2
    \lesssim \norm[\hat F]{\hat\nabla^l\hat v}^2
    + \sum_{j=1}^r \frac{1}{(j-l)!^2} \norm[\hat F]{\hat{\partial}_{\n,\xi}^j \hat v}^2.
  \end{align}

  Let $\pi_{\hat{F}}\: \hat{K} \to \hat{F}$
pro\-ject every
  point $\hat x \in \hat K$ to the closest point on $\hat F$
  (see \cref{fig:lem_stab_gen}).
  Then, we have by Taylor expansion
  \begin{align}
    \label{taylor2}
    \hat{\nabla}^l\hat{v}(\hat{x})
    = \hat{\nabla}^l \hat{v}(\pi_{\hat{F}}(\hat{x}))
    + \sum_{j=1}^r \frac{1}{j!} \hat{\partial}_1^j \hat{\nabla}^l
    \hat{v}(\pi_{\hat{F}}(\hat{x})) (\hat{x}_1 - \pi_{\hat{F}}(\hat{x})_{1})^j,
  \end{align}
  where $\hat{x}_1$ and $\pi_{\hat{F}}(\hat{x})_1$ denote
  the horizontal components of the corresponding points.

  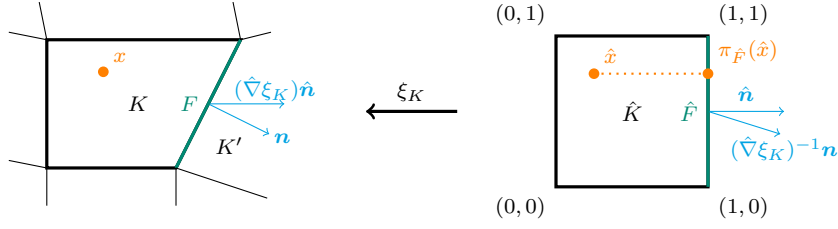
\begin{figure}
  \centering
  \begin{tikzpicture}[font = \scriptsize, scale = 1.0]
    \coordinate (A1) at (-0.2, 0.5);
    \coordinate (A2) at (1.5, 0.5);
    \coordinate (B1) at (-0.2, 2.2);
    \coordinate (B2) at (2.35, 2.2);

    \draw[very thick] (A1) -- (A2) -- (B2) -- (B1) -- cycle;

    \node at (1.0, 1.35) {\textbf{$K$}};
    \node at (2.2, 0.8) {\textbf{$K'$}};
 \ifXcolor   \draw[anchor=east] (1.925, 1.35) node {\textcolor{PineGreen}{$F$}};
    \draw[color=Cerulean, ->] (1.925, 1.35) -- ++(0.8, -0.4);
    \draw (2.9, 0.9) node {\textcolor{Cerulean}{$\n$}};
    \draw[color=Cerulean, ->] (1.925, 1.35) -- ++(1.0, 0.0);
    \draw (2.8, 1.55) node {\textcolor{Cerulean}{$(\hat\nabla \xi_K) \hat\n$}};
    \draw[very thick, PineGreen] (A2) -- (B2);
    \fill[orange] (0.54375, 1.775) circle (2pt);
    \draw[anchor=south west] (0.54375, 1.775) node {\textcolor{orange}{$x$}};

    \draw (A1) -- ++(-0.5, 0.1);
    \draw (A1) -- ++(0.0, -0.5);

    \draw (A2) -- ++(0.0, -0.5);
    \draw (A2) -- ++(1.2, -0.4);

    \draw (B1) -- ++(-0.5, 0.1);
    \draw (B1) -- ++(-0.1, 0.5);

    \draw (B2) -- ++(-0.1, 0.5);
    \draw (B2) -- ++(0.4, 0.1);

    \draw[very thick, <-] (4.0, 1.25) -- (5.2, 1.25);
    \draw[anchor=south] (4.6, 1.25) node {$\xi_K$};

    \draw[very thick] (6.5, 0.25) rectangle (8.5, 2.25);
    \draw[anchor=north east] (6.5, 0.25) node {\scriptsize $(0,0)$};
    \draw[anchor=south east] (6.5, 2.25) node {\scriptsize $(0,1)$};
    \draw[anchor=north west] (8.5, 0.25) node {\scriptsize $(1,0)$};
    \draw[anchor=south west] (8.5, 2.25) node {\scriptsize $(1,1)$};
    \node at (7.5, 1.25) {\textbf{$\hat{K}$}};
    \draw[color=Cerulean, ->] (8.5, 1.25) -- ++(1.0, 0.0);
    \draw[anchor=south] (9.0, 1.25) node {\textcolor{Cerulean}{$\hat\n$}};
    \draw[color=Cerulean, ->] (8.5, 1.25) -- ++(0.95, -0.3);
    \draw (9.5, 0.75) node {\textcolor{Cerulean}{$(\hat\nabla \xi_K)^{-1} \n$}};
    \draw[very thick, PineGreen] (8.5, 0.25) -- (8.5, 2.25);
    \draw[anchor=east] (8.5, 1.25) node {\textcolor{PineGreen}{$\hat{F}$}};

    \fill[orange] (7.0, 1.75) circle (2pt);
    \draw[anchor=south west] (7.0, 1.75) node {\textcolor{orange}{$\hat{x}$}};
    \fill[orange] (8.5, 1.75) circle (2pt);
    \draw[anchor=south west] (8.5, 1.75) node {\textcolor{orange}{$\pi_{\hat{F}}(\hat{x})$}};\fi
    \draw[dotted, orange, thick] (7.0, 1.75) -- (8.5, 1.75);
  \end{tikzpicture}
  \caption{Example of a cell $K$
      with edge $F = \cl K \cap \cl K'$ transformed by $\xi_K$
      from the reference unit square $\hat{K}$ with vertical edge $\hat{F}$,
      and a visualization of the (normal) vectors
      $\n$, $\hat\n$, $(\hat\nabla \xi_K) \hat\n$, $(\hat\nabla \xi_K)^{-1} \n$
      and the projection $\pi_{\hat{F}}$ of a point $\hat{x} \in \hat{K}$
      used in the proof of \cref{lem_stab_gen}.}
  \label{fig:lem_stab_gen}
  \end{figure}

  We take the squares in \eqref{taylor2} and estimate
  \begin{align*}
    \hat{\nabla}^l \hat{v}(\hat{x})^2
    \le c_1 \biggl[
    \hat{\nabla}^l \hat{v}(\pi_{\hat{F}}(\hat{x}))^2
    + \sum_{j=1}^r \frac{1}{(j!)^2} \bigl(
    \hat{\partial}_1^j \hat{\nabla}^l
    \hat{v}(\pi_{\hat{F}}(\hat{x})) \bigr)^2
    (\hat{x}_{1} - \pi_{\hat{F}}(\hat{x})_{1})^{2j} \biggr]
  \end{align*}
  with a constant $c_1 > 0$ and with
  $(\hat{x}_{1} - \pi_{\hat{F}}(\hat{x})_{1})^2 \le 1$.
  Finally, we integrate in horizontal and vertical directions to obtain
  \begin{align}\label{est42}
    \int_{\hat{K}} |\hat{\nabla}^l\hat{v}(\hat{x})|^2 \diff\hat{x}
    \le c_1 \biggl[
    \int_{\hat{F}} |\hat{\nabla}^l \hat{v}(\hat{s})|^2 \diff\hat{s}
    + \sum_{j=1}^r \frac{1}{(j!)^2} \int_{\hat{F}}
    |\hat{\partial}_1^j \hat{\nabla}^l \hat{v}(\hat{s})|^2 \diff\hat{s}
    \biggr].
  \end{align}
  The horizontal integral does not appear on the right-hand side,
  as $\pi_{\hat{F}}(\hat{x})$ is independent of the horizontal coordinate. Due to the fact that $\hat{\partial}_1^{r+1}\hat{v}$ vanishes, the last term in the sum can be reduced to
   \begin{align}
\begin{split}
   \label{est422}
    \int_{\hat{K}} |\hat{\nabla}^l\hat{v}(\hat{x})|^2 \diff\hat{x}
    \le c_1 \biggl[
    \int_{\hat{F}} |\hat{\nabla}^l \hat{v}(\hat{s})|^2 \diff\hat{s}
    + \sum_{j=1}^{r-l} &\left(\frac{1}{(j!)^2} \int_{\hat{F}}
    |\hat{\partial}_1^j \hat{\nabla}^l \hat{v}(\hat{s})|^2 \diff\hat{s}
    \right) \\
    &+ \frac{1}{(r!)^2} \int_{\hat{F}}
    |\hat{\partial}_1^r \hat{\partial}_2^l \hat{v}(\hat{s})|^2 \diff\hat{s}
    \biggr].
    \end{split}
  \end{align}

  Due to the shape regularity of the mesh,
  the tangential vector $\hat{e}_2$ of $F$
  and the perturbed normal vector $\hat\n_{\xi}(\hat{x})$
  are linearly independent for each $\hat{x}\in\hat{F}$, which allows us to express each coordinate vector $e_k$ ($k=1,2$) as
  \begin{align*}
e_k=\alpha_1(\hat{x}) \hat{\n}_{\xi}(\hat{x}) + \alpha_2(\hat{x}) e_2
\end{align*}
with {coefficient functions} $\alpha_1(\hat{x}), \alpha_2(\hat{x})$ {satisfying uniformly} $|\alpha_1(\hat{x})|,\,|\alpha_2(\hat{x})|\leq c$.

  Thus, we have for the partial derivatives
\begin{align}\label{eq:partder}
\hat{\partial}_k\hat{v}=\hat{\nabla} \hat{v} \cdot e_k = \alpha_1 \hat{\nabla} \hat{v} \cdot \hat{\n}_{\xi}(\hat{x}) + \alpha_2 \hat{\nabla} \hat{v} \cdot e_2 = \alpha_1 \hat{\partial}_{\n,\xi} \hat{v} + \alpha_2 \hat{\partial}_2 \hat{v},
\end{align}
where (trivially) $\alpha_1=0$ for $k=2$.
{{We use the same argument iteratively to get an expression for the higher derivatives. Note that both the coefficients $\alpha_i$ and $\hat\n_\xi$ may depend on $\hat{x}$.}}
 {Differentiating
 $\hat\n_\xi=(\hat\nabla\xi_K)^{-1}\n/\|(\hat\nabla\xi_K)^{-1}\n\|$ gives,
 for $k=1,2$,
 \begin{align*}
   \hat\partial_k\hat\n_\xi
   =-(I-\hat\n_\xi\hat\n_\xi^\top)
     (\hat\nabla\xi_K)^{-1}\bigl(\hat\partial_k(\hat\nabla\xi_K)\bigr)\hat\n_\xi\,
   \Rightarrow\, \|\hat\partial_k\hat\n_\xi\|\lesssim1.
 \end{align*}
 Here, we have used that shape regularity and bilinearity give
 $\|(\hat\nabla\xi_K)^{-1}\|\lesssim h^{-1}$ and
 $\|\hat\partial_k(\hat\nabla\xi_K)\|\lesssim h$.
 Further differentiation yields uniform bounds for {derivatives
 of any fixed order, see~\cite{ArbogastTaoWang2022}, Section~6.2. A similar argumentation gives the boundedness for derivatives of the coefficients $\alpha_1,\alpha_2$.}
 After substituting the representation of $\hat\partial_1$ from~\eqref{eq:partder},
 we move all tangential derivatives to the left using
 \begin{align*}
   \hat\partial_2(\hat\partial_{\n,\xi}\hat v)
   =\hat\partial_{\n,\xi}(\hat\partial_2\hat v)
     +(\hat\partial_2\hat\n_\xi)\cdot\hat\nabla\hat v
   \quad\Rightarrow\quad
   \hat\partial_{\n,\xi}(\hat\partial_2\hat v)
   =\hat\partial_2(\hat\partial_{\n,\xi}\hat v)
     -(\hat\partial_2\hat\n_\xi)\cdot\hat\nabla\hat v.
 \end{align*}
 Including the resulting lower-order terms, we obtain}
  for $k\in\{1,2\}$, $l\in\{0,1\}$ {and $j\ge1$}
  {\begin{align}\label{eq:higher-partder}
  \hat{\partial}_1^j \hat{\partial}_k^l \hat{v}(\hat{s}) = \sum_{m=1}^{j+l}\sum_{i=0}^{m} \tilde{\alpha}_{m,i}(\hat{s}) \hat{\partial}_2^{m-i}\hat{\partial}_{\n,\xi}^i \hat{v}(\hat{s})
  \end{align}}%
  with {uniformly bounded coefficient functions $\tilde{\alpha}_{m,i}$}.
  Due to the fact that $\alpha_1=0$ for $k=2$ in~\eqref{eq:partder},
    we have in the special case $j=r$, $l=1$, $k=2$ also
  \begin{align}\label{est423}
  {\hat{\partial}_1^r \hat{\partial}_2 \hat{v}(\hat{s}) = \sum_{m=1}^{r+1}\sum_{i=0}^{\min\{m,r\}} \tilde{\beta}_{m,i}(\hat{s}) \hat{\partial}_2^{m-i}\hat{\partial}_{\n,\xi}^i \hat{v}(\hat{s}),}
  \end{align}
  {with uniformly bounded coefficients $\tilde{\beta}_{m,i}$,}
  i.e., no directional derivatives of order $r+1$ in direction $\hat{\n}_{\xi}$
  appear on the right-hand side.
  This will be used to estimate the last term in~\eqref{est422}
  by terms involving derivatives of maximum order $r$
  in direction $\hat{\n}_{\xi}$.

  For the sum on the right-hand side of~\eqref{est422} this yields {{u}sing~\eqref{eq:higher-partder}}
  \begin{align}\label{est43}
  \sum_{j=1}^{r-l} \frac{1}{(j!)^2} \int_{\hat{F}}
    |\hat{\partial}_1^j \hat{\nabla}^l \hat{v}(\hat{s})|^2 \diff\hat{s} \leq c \sum_{j=1}^{r-l} \frac{1}{(j!)^2} {\sum_{m=1}^{j+l}\sum_{i=0}^{m}} \int_{\hat{F}}
    {|\hat{\partial}_2^{m-i}\hat{\partial}_{\n,\xi}^i \hat{v}(\hat{s})|^2} \diff\hat{s}.
  \end{align}
  For the last term in~\eqref{est422}, we obtain {{f}or $l=1$} using~\eqref{est423}
  \begin{align}\label{est431}
  \frac{1}{(r!)^2} \int_{\hat{F}}
    |\hat{\partial}_1^r \hat{\partial}_2^l \hat{v}(\hat{s})|^2 \diff\hat{s}
     \lesssim \frac{1}{(r!)^2}{\sum_{m=1}^{r+1}\sum_{i=0}^{\min\{m,r\}}} \int_{\hat{F}} {|\hat{\partial}_2^{m-i}\hat{\partial}_{\n,\xi}^i \hat{v}(\hat{s}) |^2} \diff\hat{s}.
    \end{align}
  {For $l=0$, using~\eqref{eq:higher-partder} with $j=r$ yields
  \begin{align}\label{est431-l0}
  \frac{1}{(r!)^2} \int_{\hat{F}}
    |\hat{\partial}_1^r \hat{v}(\hat{s})|^2 \diff\hat{s}
    \lesssim \frac{1}{(r!)^2}\sum_{m=1}^{r}\sum_{i=0}^{m} \int_{\hat{F}}
    |\hat{\partial}_2^{m-i}\hat{\partial}_{\n,\xi}^i \hat{v}(\hat{s})|^2 \diff\hat{s}.
  \end{align}}
   Next, we use that both $\norm{\hat{\varphi}}_{\hat{F}}$ and
  $\norm{\hat{\varphi}}_{\tilde{H}_2^m(\hat{F})} :=
  \bigl(\sum_{i=0}^m \norm{\hat{\partial}_2^i
    \hat{\varphi}}_{\hat{F}}^2 \bigr)^{1/2}$
  define norms on {the finite-dimensional spaces
  $W_{i,K} :=
  \defset{(\hat{\partial}_{\n,\xi}^i\hat w)|_{\hat F}}{\hat w\in Q_r(\hat K)}$.
  Using norm equivalence, we obtain}
  \begin{align*}
  \int_{\hat{F}}
    {|\hat{\partial}_2^{m-i}\hat{\partial}_{\n,\xi}^i \hat{v}(\hat{s})|^2} \diff\hat{s} \lesssim  \|\hat{\partial}_{\n,\xi}^i \hat{v}\|_{\tilde{H}_2^{{m}}(\hat F)}^2 \lesssim  \|\hat{\partial}_{\n,\xi}^i \hat{v}\|_{\hat F}^2 .
  \end{align*}
We will use this estimate for $(i,l)\neq (0,1)$. For $(i,l)=(0,1)$ we obtain similarly
  \begin{align*}
 {\int_{\hat{F}}
    |\hat{\partial}_2^m \hat{v}(\hat{s})|^2 \diff\hat{s} \lesssim \|\hat{\partial}_2 \hat{v}\|_{\hat F}^2,\qquad m\ge1.}
    \end{align*}
 Substituting into~\eqref{est43}, we obtain
  \begin{align}
  \begin{split}\label{est444}
  \sum_{j=1}^{r-l} \frac{1}{(j!)^2} \int_{\hat{F}}
    |\hat{\partial}_1^j \hat{\nabla}^l \hat{v}(\hat{s})|^2 \diff\hat{s} \,&\lesssim \, \sum_{j=1}^{r-l} \frac{1}{(j!)^2} \biggl( \|\hat{\partial}_2^l \hat{v}\|_{\hat F}^2+ \sum_{i=1}^{j+l} \int_{\hat{F}}
    |\hat{\partial}_{\n,\xi}^i\hat{v}|^2 \diff\hat{s}\biggr)\\ &\lesssim \|\hat{\partial}_2^l \hat{v}\|_{\hat F}^2+\sum_{j=1}^{r} \frac{1}{(j-l)!^2}
    \|\hat{\partial}_{\n,\xi}^j\hat{v}\|_{\hat{F}}^2.
    \end{split}
  \end{align}
To obtain the scaling within the sum on the last line, we have used that by a re-ordering of the sums and the fact that {$\sum_{j=\max\{p,1\}}^r
\frac{1}{(j!)^2} \leq \frac{2}{(p!)^2}$ for $p\ge0$,}\begin{align*}
\sum_{j=1}^{r-l} \sum_{i=1}^{j+l} \frac{1}{(j!)^2} \int_{\hat{F}}
    |\hat{\partial}_{\n,\xi}^i\hat{v}|^2 \diff\hat{s} &= \sum_{i=1}^{r} \sum_{j=\max\{i-l,1\}}^{r-l} \frac{1}{(j!)^2} \int_{\hat{F}}
    |\hat{\partial}_{\n,\xi}^i\hat{v}|^2 \diff\hat{s}\\
    &\leq \sum_{i=1}^{r} \frac{2}{((i-l)!)^2} \int_{\hat{F}}
    |\hat{\partial}_{\n,\xi}^i\hat{v}|^2 \diff\hat{s}.
\end{align*}
From~\eqref{est431} { and~\eqref{est431-l0}}, we obtain similarly to~\eqref{est444}
 \begin{align*}
  \frac{1}{(r!)^2} \int_{\hat{F}}
    |\hat{\partial}_1^r \hat{\partial}_2^l \hat{v}(\hat{s})|^2 \diff\hat{s}
     &\lesssim \|\hat{\partial}_2^l \hat{v}\|_{\hat F}^2 + \frac{1}{(r!)^2}\sum_{i=1}^r \int_{\hat{F}} |\hat{\partial}_{\n,\xi}^i \hat{v}(\hat{s}) |^2 \diff\hat{s} \\
     &\leq \|\hat{\partial}_2^l \hat{v}\|_{\hat F}^2 + \sum_{i=1}^r \frac{1}{((i-l)!)^2} \int_{\hat{F}} |\hat{\partial}_{\n,\xi}^i \hat{v}(\hat{s}) |^2 \diff\hat{s}.
     \end{align*}
   Hence, it follows from~\eqref{est422}:
  \begin{align}\label{eq:new4}
    \int_{\hat{K}} |\hat{\nabla}^l\hat{v}(\hat{x})|^2 \diff\hat{x}
     \lesssim \biggl[
    \int_{\hat{F}} |\hat{\nabla}^l \hat{v}(\hat{s})|^2 \diff\hat{s}
    {}+ \sum_{j=1}^{r} \frac{1}{(j-l)!^2}
    \|\hat{\partial}_{\n,\xi}^j\hat{v}\|_{\hat{F}}^2  \biggr].
  \end{align}
  This shows \eqref{hatF2}.
 Standard estimates based on an integral transformation then yield,
  using~\eqref{normalDer} and the fact that $\|(\hat\nabla \xi_K)^{-1} \n\|^{{-1}}\lesssim h$,
  \begin{align}
    \begin{split}
      \norm{&\nabla^l v}_K^2\label{pK_pF}
      \lesssim h^{-2l} |K| \, \norm{\hat{\nabla}^l \hat{v}}_{\hat{K}}^2
      \lesssim h^{-2l} |K| \biggl[
        \norm{\hat{\nabla}^l \hat{v}}_{\hat{F}}^2
        + \sum_{j=1}^r \frac{1}{(j-l)!^2}
        \norm{\hat{\partial}_{\n,\xi}^j \hat{v}}_{\hat{F}}^2 \biggr] \\
      &\lesssim \frac{|K|}{|F|} \biggl[
        \norm{\nabla^l v}_{F}^2
        + \sum_{j=1}^r \frac{h^{2(j-l)}}{(j-l)!^2} \norm{\partial_{\n}^j v}_{F}^2 \biggr]
      \lesssim \biggl[ h \norm{\nabla^l v}_{F}^2
        + \sum_{j=1}^r \frac{h^{2(j-l)+1}}{(j-l)!^2} \norm{\partial_{\n}^j v}_{F}^2
      \biggr].
    \end{split}
  \end{align}
  {We note that also in the back-transformation of higher-order normal derivatives in the second line of~\eqref{pK_pF}, the dependence
  of $\hat\n_\xi$ and the scaling factor in~\eqref{normalDer} on
  $\hat x$ have been taken into account. By repeated application
  of~\eqref{normalDer} and the product rule, the resulting lower-order normal derivatives are bounded
  by the terms in the sum on the right-hand side.}

  {Next,} let $v' := v^h|_{K'}$.
  Using $\nabla^j v|_F = \jump{\nabla^j v^h} + \nabla^j v'|_F$
  and, for $l=0$, the fact that $v^h$ is continuous across elements,
  it follows that
  \begin{equation}
    \label{est_ph2}
    \begin{aligned}
      \norm[K]{\nabla^l v^h}^2
      &\lesssim  \biggl[ h \left(
        \norm[F]{\jump{\nabla^l v^h}}^2 + {} \norm[F]{\nabla^l v'}^2
      \right)
      + \sum_{j=1}^r \frac{h^{2(j-l)+1}}{(j-l)!^2} \bigl(
      \norm[F]{\jump{\partial_{\n}^j v^h}}^2
      + \norm[F]{\partial_{\n}^j v'}^2 \bigr) \biggr] \\
      &\lesssim \biggl[ \norm[K']{\nabla^l v^h}^2
      + \sum_{j=1}^r \frac{h^{2(j-l)+1}}{(j-l)!^2}
      \norm[F]{\jump{\partial_{\n}^j v^h}}^2 \biggr].
    \end{aligned}
  \end{equation}
  Here, the second estimate
  follows from the inverse inequalities
  \begin{align*}
    \norm{\nabla^j v'}_{F}^2
    \le c h^{2(l-j)-1} \norm{\nabla^l v'}_{K'}^2
    \quad \text{for } j \ge l,
  \end{align*}
  which also hold when replacing $\nabla^j$ by $\partial_{\n}^j$.
  Now, let $K \subseteq G^h$ be an interface cell.
  Due to Assumption~\ref{FKK} there exists a cell $K_i \subseteq \interf$
  that can be reached from $K$
  by crossing a finite number $N_\F$ of edges in~$\Oih$,
  where $N_\F$ is independent of the mesh size $h$.
  Let $\F_{K K_f}$ denote the set of these edges.
  By repeated application of estimate \eqref{est_ph2} with the corresponding constant $c_v>0$, we obtain
  \begin{align*}
    \norm[K]{\nabla^l v^h}^2
    \le \max\{c_v^{N_\F},c_v\} \biggl[ \norm[K_i]{\nabla^l v^h}^2
    + \sum_{F \in \F_{K K_f}} \sum_{j=1}^r \frac{h^{2(j-l)+1}}{(j-l)!^2}
    \norm[F]{\jump{\partial_{\n}^j v^h}}^2 \biggr].
  \end{align*}
  Since the number of times a cell $K_i \in \interf$
  is the final element of the paths among all $K \subseteq G^h$
  is bounded by Assumption~\ref{FKK} by, say $N_{\Omega_i}$,
  summing over all $K \subseteq G^h$ gives
  \begin{align}
  \begin{split}\label{finalghost}
    \norm[G^h]{\nabla^l v^h}^2
    \le N_{\Omega_i} \max\{c_v^{N_\F},c_v\} \biggl[
    \norm[\interf]{\nabla^l v^h}^2
    + \sum_{F \in \F_G^i} \sum_{j=1}^r \frac{h^{2(j-l)+1}}{(j-l)!^2}
    \norm[F]{\jump{\partial_{\n}^j v^h}}^2 \biggr].
    \end{split}
  \end{align}
  As $w(\kappa_i)$ is bounded below by $\frac12\wmax^{-1}$,
  this implies \eqref{ghost_gen}.
\end{proof}

\subsection{Stability estimate in the energy norm}

We define the following seminorms
for the stability and error analysis:
\begin{align*}
  \ET(U^h)^2
  &:= \frac{\rho_f}{2} \normOf{v_f^h}^2
    + \frac{\rho_s}{2} \normOsh{v_s^h}^2
    + \mu_s \normOsh{\nabla u^h}^2
    {{}+ \frac{\lambda_s}{2} \normOs{\tr{\nabla u^h}}^2}, \\
  \tracenorm{\phi^h}
  &:= h^{-\frac12} \normGi{\phi^h}, \\
  \normiii{U^h}^2 %
  &:= \rho_f \nu_f \normOfh{\nabla v_f^h}^2
  + \rho_f \nu_f \gamma_N \tracenorm{v_f^h-v_s^h}^2 + g_p^{h,w}(p^h,p^h).
\end{align*}
We make the following assumption for the right-hand side functionals $F^n$.
\begin{assumption}%
  \label{ass:F}%
  There exists $\CF > 0$ such that
  the following estimate holds {for every $m=1,\ldots,N$} and for
  $\Phi^{h,n} = (\phi_f^{h,n}, \xi^{h,n}, \phi_s^{h,n}, \psi^{h,n})$ in $\X^h$,
  $n = 1, \dots, N$\ with
  $\pres := \phi_s^{h,n} - k^{-1} \delta \psi^{h,n}$:
  \begin{align}
    \label{F:ass}
    \sum_{n=1}^{{m}} k &F^n(\Phi^{h,n})
    \lesssim \CF
    \biggl[ \ET(\Phi^{h,{m}})^2 +\\
    &\sum_{n=1}^{{m}} k \bigl( \normiii{\Phi^{h,n}}^2 +
      h^2\normOf{\nabla \xi^{h,n}}^2 + \normOsh{\phi_s^{h,n}}^2 + \normOsh{\nabla \psi^{h,n}}^2 +
      \normOsh{\nabla\pres}^2 \bigr) \biggr]^{\frac12}.
      \notag
  \end{align}
\end{assumption}
For the right-hand side of \cref{strong_form},
$F^n(\Phi^h) = \rho_f \prodOf{f{_f}(t_n)}{\phi_f^h}
+ \rho_s \prodOs{f{_s}(t_n)}{\phi_s^h}$,
this assumption is satisfied by the Cauchy--Schwarz and Poincar\'e inequalities  with
\begin{align*}
  \sum_{n=1}^N k F^n(\Phi^{h,n})
  &\lesssim \CF \biggl[ \sum_{n=1}^N k\bigl(\normOf{\nabla \phi^{h,n}_f}^2 + \normOs{\phi^{h,n}_s}^2\bigr)\biggr]^{\frac12} \\
  &\leq \CF \biggl[ \sum_{n=1}^N k \bigl( \normiii{\Phi^{h,n}}^2 + \normOs{\phi_s^{h,n}}^2 \bigr) \biggr]^{\frac12}
\end{align*}
using $\CF^2 = \sum_{n=1}^N k
\bigl( \norm[\Omega_f]{\rho_f f{_f}(t_n)}^2 + \normOs{\rho_s f{_s}(t_n)}^2 \bigr)$, where we have used that $\Gamma_f^D \neq \emptyset$.
In the final a priori error analysis in \cref{sec_apriori},
we will show and use~\eqref{F:ass} also for a different right-hand side
functional consisting of interpolation and consistency errors.

\begin{theorem}[Stability estimate for the fully discretized
    bilinear form\ $\A^{kh}$]%
  \label{theo_stab}%
  Let $T > 0$ and let $t_n = n k$ for $n = 0, \dots, N$.
  {If $\gamma_N>0$ is sufficiently large and $k$ is sufficiently small,
  there exists a constant $C>0$, independent of $h$ and $k$,
  such that, for all $U^{h,n}=(v_f^{h,n},p^{h,n},v_s^{h,n},u^{h,n})\in\X^h$
  and $\vres=v_s^{h,n}-k^{-1}\delta u^{h,n}$, $n=1,\ldots,N$,}
  \begin{equation}
    \label{stab_est0}
    {\begin{aligned}
      \frac1C\biggl[&\frac{\rho_f}{2}\normOf{v_f^{h,N}}^2
      +\frac{\rho_s}{2}\normOs{v_s^{h,N}}^2
      +\frac{\rho_s}{2}g_{v_s}^{h,w}(v_s^{h,N},v_s^{h,N})\\
      &+\mu_s\normOs{\symgrad{u^{h,N}}}^2
      +\frac{\lambda_s}{2}\normOs{\tr{\nabla u^{h,N}}}^2
      +\mu_s g_u^{h,w}(u^{h,N},u^{h,N})\\
      &+\sum_{n=1}^N\Bigl(k\normiii{U^{h,n}}^2
      + \ET(\delta U^{h,n})^2
      \Bigr)\biggr]\\
      \leq{}&\sum_{n=1}^N\A^{kh}(U^{h,n},U^{h,n-1};U^{h,n}) \\[-\jot]
      &+C\biggl[\ET(U^{h,0})^2+\sum_{n=1}^N k\Bigl(\norm[H^1(\Osh)]{\vres}^2
      +\ET(U^{h,n})^2\Bigr)\biggr].
    \end{aligned}}
  \end{equation}
\end{theorem}

\begin{proof}
  We consider the contributions to
  $\A^{kh}$ term by term.
  First, it holds for the discrete time derivatives $\delta v_i^{h,n}$
  by using a telescoping equality
  for $i \in \set{f, s}$:
  \begin{equation}
    \label{disctime}
    \rho_i \iprod[\Omega_i]{\delta v_i^{h,n}}{v_i^{h,n}} =
    \frac12 \rho_i \bigl(
    \normOi{v_i^{h,n}}^2
    + \normOi{\delta v_i^{h,n}}^2
    - \normOi{v_i^{h,n-1}}^2 \bigr).
  \end{equation}
  The fluid bulk terms and the Nitsche terms can be estimated
  as in \cite[Lem.~3.1]{BURMAN2014FSI},
  \begin{equation}
    \label{fluidbulk}
    \begin{split}
      a_f^h(U^{h,n}&, U^{h,n}) + S_f^h(U^{h,n}, U^{h,n}) +  j^h(U^{h,n}, U^{h,n}) \\
      &\gtrsim \rho_f \nu_f \normOfh{\nabla v_f^{h,n}}^2
      + g_p^{h,w}(p^{h,n},p^{h,n})
      + \rho_f \nu_f \gamma_N \tracenorm{v_f^{h,n} - v_s^{h,n}}^2.
    \end{split}
  \end{equation}
It remains to
estimate the solid terms $a_s^h(U^{h,n}, U^{h,n})$.
  We split the solid stress into
  \begin{equation}
    \begin{aligned}%
      \label{newesti}
      k \prodOs{\sigma_s(u^{h,n})&}{\nabla v_s^{h,n}}
      = 2 k \mu_s \prodOs{\symgrad{u^{h,n}}}{\nabla v_s^{h,n}}
      + k \lambda_s
      \prodOs{\tr{\symgrad{u^{h,n}}} \mathbb I}{\nabla v_s^{h,n}} \\
      &= 2 \mu_s
      \prodOs{\symgrad{u^{h,n}}}{\symgrad{\delta u^{h,n}}}
      + \lambda_s
      \prodOs{\tr{\nabla u^{h,n}}}{\tr{\nabla \delta u^{h,n}}}
      \\[-2pt]
      &\ + 2 k \mu_s \prodOs{\symgrad{u^{h,n}}}{\nabla\vres}
      + k \lambda_s \prodOs{\tr{\nabla u^{h,n}} \mathbb I}{\nabla\vres}.
    \end{aligned}
  \end{equation}
  For the last line in \eqref{newesti},
  the Cauchy--Schwarz and Young inequalities yield
  \begin{equation}
    \begin{aligned}
      \label{newesti2}
      2 k \mu_s &\prodOs{\symgrad{u^{h,n}}}{\nabla\vres}
      + k \lambda_s \prodOs{\tr{\nabla u^{h,n}} \mathbb I}{\nabla\vres} \\
      &\ge -c_0 k \Bigl(
      \mu_s \normOs{\symgrad{u^{h,n}}}^2
      + \frac{\lambda_s}{2} \normOs{\tr{\nabla u^{h,n}}}^2 \Bigr)
      - c_0 k \normOs{\nabla\vres}^2
    \end{aligned}
  \end{equation}
  with some $c_0 > 0$.
  For the second line in \eqref{newesti}, we use a telescoping equality,
  \begin{equation}
    \label{esti3}
    \begin{aligned}
      2 \mu_s \prodOs{\symgrad{u^{h,n}}&}{\symgrad{\delta u^{h,n}} }
      + \lambda_s
      \prodOs{\tr{\nabla u^{h,n}}}{\tr{\nabla \delta u^{h,n}}}
      \\
      &= \mu_s \bigl( \normOs{\symgrad{u^{h,n}}}^2
      + \normOs{\symgrad{\delta u^{h,n}}}^2
      - \normOs{\symgrad{u^{h,n-1}}}^2 \bigr) \\[-2pt]
      &\ + \frac{\lambda_s}{2}
      \bigl( \normOs{\tr{\nabla u^{h,n}}}^2
      + \normOs{\tr{\nabla \delta u^{h,n}}}^2
      - \normOs{\tr{\nabla u^{h,n-1}}}^2 \bigr).
    \end{aligned}
  \end{equation}
  Combining \eqref{newesti}--\eqref{esti3} yields
  \begin{align*}
    k \prodOs{\sigma_s &(u^{h,n})}{\nabla v_s^{h,n}} \\
    &\ge \mu_s \left( (1 - c_0 k) \normOs{\symgrad{u^{h,n}}}^2
    + \normOs{\symgrad{\delta u^{h,n}}}^2
    - \normOs{\symgrad{u^{h,n-1}}}^2 \right)
    - c_0 k \normOs{\nabla\vres}^2 \\[-2pt]
    &+ \frac{\lambda_s}{2} \left( (1 - c_0 k) \normOs{\tr{\nabla u^{h,n}}}^2
    + \normOs{\tr{\nabla \delta u^{h,n}}}^2
    - \normOs{\tr{\nabla u^{h,n-1}}}^2 \right).
  \end{align*}
  For the stabilization terms, we have by a similar argumentation
  \begin{equation}
    \begin{aligned}
      \label{stabsolid}
      \rho_s g_{v_s}^{h,w}&(\delta v_s^{h,n}, v_s^{h,n})
      + 2 k \mu_s g_u^{h,w}(u^{h,n}, v_s^{h,n}) \\
      &= \rho_s g_{v_s}^{h,w}(\delta v_s^{h,n}, v_s^{h,n})
      + 2 \mu_s g_u^{h,w}(u^{h,n}, \delta u^{h,n})
      + 2 k \mu_s g_u^{h,w}(u^{h,n}, \vres) \\
      &\ge \frac{\rho_s}{2} \left( g_{v_s}^{h,w}(v_s^{h,n},v_s^{h,n})
      + g_{v_s}^{h,w}(\delta v_s^{h,n}, \delta v_s^{h,n})
      - g_{v_s}^{h,w}(v_s^{h,n-1},v_s^{h,n-1}) \right) \\[-2pt]
      &\quad+ \mu_s \left( (1 - c_0 k) g_u^{h,w}(u^{h,n}, u^{h,n})
      + g_u^{h,w}(\delta u^{h,n}, \delta u^{h,n}) \right) \\[-2pt]
      &\quad- \mu_s g_u^{h,w}(u^{h,n-1}, u^{h,n-1})
        - c_0 k \normOsh{\nabla\vres}^2.
    \end{aligned}
  \end{equation}
  Combining \eqref{fluidbulk}--\eqref{stabsolid},
  we arrive at the following estimate
  (with some constants $c_0, c_1, c_2 > 0$):
    \begin{align}
\begin{split}
    \label{longbefGronw}
  	&c_1 k \normiii{U^{h,n}}^2
  	+ \frac{\rho_f}{2} \normOf{\delta v_f^{h,n}}^2
  	+ \frac{\rho_s}{2} \normOs{\delta v_s^{h,n}}^2
  	+ \frac{\rho_s}{2} g_{v_s}^{h,w}(\delta v_s^{h,n}, \delta v_s^{h,n})\\[-2pt]
  	&\quad+ \mu_s \normOs{\symgrad{\delta u^{h,n}}}^2
  	+ \frac{\lambda_s}{2} \normOs{\tr{\nabla \delta u^{h,n}}}^2
  	+ \mu_s g_u^{h,w}(\delta u^{h,n}, \delta u^{h,n}) \\[-2pt]
  	&\quad + \frac{\rho_f}{2} \normOf{v_f^{h,n}}^2
  	+ \frac{\rho_s}{2} \normOs{v_s^{h,n}}^2
  	+ \frac{\rho_s}{2} g_{v_s}^{h,w}(v_s^{h,n}, v_s^{h,n}) \\[-2pt]
  	&\quad+ (1 - c_0 k) \Bigl( \mu_s \normOs{\symgrad{u^{h,n}}}^2
  	+ \frac{\lambda_s}{2} \normOs{\tr{\nabla u^{h,n}}}^2
  	+ \mu_s g_u^{h,w}(u^{h,n}, u^{h,n}) \Bigr) \\[-2pt]
  	&\le c_2\A^{kh}(U^{h,n},U^{h,n-1}; U^{h,n}) + c_0 k \normOsh{\nabla\vres}^2 \\[-2pt]
  	&\quad+ \frac{\rho_f}{2} \normOf{v_f^{h,n-1}}^2
  	+ \frac{\rho_s}{2} \normOs{v_s^{h,n-1}}^2
  	+ \frac{\rho_s}{2} g_{v_s}^{h,w}(v_s^{h,n-1}, v_s^{h,n-1}) \\[-2pt]
  	&\quad + \mu_s \normOs{\symgrad{u^{h,n-1}}}^2
  	+ \frac{\lambda_s}{2} \normOs{\tr{\nabla u^{h,n-1}}}^2
  	+ \mu_s g_u^{h,w}(u^{h,n-1}, u^{h,n-1})
  	.\end{split}
  	\end{align}
    {By Korn's second inequality, there exist constants
    $c_K,C_K>0$ such that
    \begin{align*}
      \normOs{\symgrad{\delta u^{h,n}}}^2
      &\geq c_K\normOs{\nabla\delta u^{h,n}}^2
        -C_K\normOs{\delta u^{h,n}}^2\\
      &=c_K\normOs{\nabla\delta u^{h,n}}^2
        -C_Kk^2\normOs{v_s^{h,n}-\vres}^2\\
      &\geq c_K\normOs{\nabla\delta u^{h,n}}^2
        -2C_Kk^2\bigl(\normOs{\vres}^2+\normOs{v_s^{h,n}}^2\bigr).
    \end{align*}}
{Then, by \eqref{est:vs} and \eqref{est:u}, there exist
constants $c_4,c_5>0$ such that}
    \begin{equation}\label{eq:korn}
      {\begin{aligned}
        &\frac{\rho_s}{2}\normOs{\delta v_s^{h,n}}^2
        +\frac{\rho_s}{2}g_{v_s}^{h,w}(\delta v_s^{h,n},\delta v_s^{h,n})
        +\mu_s\normOs{\symgrad{\delta u^{h,n}}}^2\\
        &\quad+\frac{\lambda_s}{2}\normOs{\tr{\nabla\delta u^{h,n}}}^2
        +\mu_s g_u^{h,w}(\delta u^{h,n},\delta u^{h,n})\\
        &\geq c_4\biggl(\frac{\rho_s}{2}\normOsh{\delta v_s^{h,n}}^2
        +\mu_s\normOsh{\nabla\delta u^{h,n}}^2
        +\frac{\lambda_s}{2}\normOs{\tr{\nabla\delta u^{h,n}}}^2\biggr)\\
        &\quad-c_5k^2\bigl(\normOs{\vres}^2+\normOs{v_s^{h,n}}^2\bigr).
      \end{aligned}}
    \end{equation}
    {Summing \eqref{longbefGronw} over $n=1,\ldots,N$
    and using \eqref{eq:korn} for the increments gives the desired
    estimate after dividing by $c_2$ and choosing $C$ sufficiently large.
    Here the terms
    $k^2\normOs{v_s^{h,n}}^2$ are bounded by
    $k\normOs{v_s^{h,n}}^2$ for $k\leq1$.
    For the initial elastic contribution we use directly
    \begin{equation*}
      \mu_s\normOs{\symgrad{u^{h,0}}}^2
      +\frac{\lambda_s}{2}\normOs{\tr{\nabla u^{h,0}}}^2\\
      \leq\mu_s\normOsh{\nabla u^{h,0}}^2
      +\frac{\lambda_s}{2}\normOs{\tr{\nabla u^{h,0}}}^2.
    \end{equation*}
    {Finally, using elementwise inverse inequalities, the
    ghost penalties $g_u^{h,w}(u^{h,n},u^{h,n})$ and
    $g_{v_s}^{h,w}(v_s^{h,n},v_s^{h,n})$ in~\eqref{longbefGronw} are bounded by
    $\normOsh{\nabla u^{h,n}}^2$ and $\normOsh{v_s^{h,n}}^2$,
    respectively, for $n=0,\ldots,N$.}
    This proves \eqref{stab_est0}.}
\end{proof}

\subsection{Stability of pressure}

Next, we prove a stability estimate for the pressure%
. Note that in \cite{BURMAN2014FSI, BurmanFernandezGerosa2023}
stability for the pressure was obtained directly
in an $H^1$-seminorm by means of the pressure stabilization term.
For the case of inf-sup stable finite elements without pressure stabilization, we can use an inf-sup condition to get control over $p^n$ in the same semi-norm, following arguments introduced recently by Neilan and Olshanskii for simplicial elements in \cite{NeilanOlshanskii2024}.

\begin{lemma}%
  \label{lem_infsup}%
  {Let Assumption~\ref{FKK} hold.}
  For each $p^h \in \P^h$, there exist $w^h \in \V_f^h$ and $\gamma_1, C > 0$ such that
  \begin{align}
    \label{phwh}
    \gamma_1 h^2 \normOfh{\nabla p^h}^2
    &\le \prodOf{\div w^h}{p^h} - \prodGi{w^h \cdot \n}{p^h}
      + g_p^{h,w}(p^h, p^h), \\
    \label{whH1}
    \norm[H^k(\Ofh)]{w^h} &\le Ch^{2-k} \normIf{\nabla p^h} \quad \text{ for } k\in\{0,1\}.
  \end{align}
  Moreover, the following inf-sup-condition holds for some $\gamma_2 > 0$:
  \begin{align}\label{infsup}
    \gamma_2 h \normOfh{\nabla p^h}
    \le \sup_{w^h \in \V_f^h}
    \frac{\prodOf{\div w^h}{p^h} - \prodGi{w^h \cdot \n}{p^h}}
    {\normOf{\nabla w^h}} + g_p^{h,w}(p^h, p^h)^{\frac12}.
  \end{align}
\end{lemma}
\begin{proof}
By the inf-sup stability of the Taylor--Hood elements on $\ifluid$
there exists $w^h \in \V_f^h$ with supp$(w^h) \subset \ifluid$
and $w^h = 0$ on $\partial \ifluid$,
such that (see, e.g., \cite[Cor.~1]{Zulehner2024})
\begin{equation*}
  \gamma_1 h^2 \normIf{\nabla p^h}^2 \le \iprod[\ifluid]{\div w^h}{p^h}, \quad
  \|w^h\|_{H^k(\ifluid)} \le ch^{2-k} \normIf{\nabla p^h}, \quad
  k \in \set{0, 1}.
\end{equation*}
{We note that this result has been shown elementwise
in~\cite[Lemma 1]{Zulehner2024} and the only required geometric condition
on $\ifluid$ is that each cell $K\subset\ifluid$ has at least one
vertex without imposed Dirichlet values, which is guaranteed by
Assumption~\ref{FKK}.}

By extending $w^h$ by zero to $\Ofh$, we obtain~\eqref{whH1} as well as
\begin{align}\label{whph1}
  \gamma_1 h^2 \normIf{\nabla p^h}^2
  &\le \prodOf{\div w^h}{p^h} - \prodGi{w^h \cdot \n}{p^h}.
\end{align}
Hence, \eqref{phwh} follows by \cref{lem_stab_gen} using $\normOfh{\nabla p^h}^2 \lesssim \normIf{\nabla p^h}^2 + h^{-2}\ghw[p]{p^h}{p^h}$. To show the inf-sup condition~\eqref{infsup}, we combine~\eqref{whph1}  and~\eqref{whH1} for $k=1$ to get
\begin{align*}
 {h \normIf{\nabla p^h}}
    &\lesssim \sup_{w^h \in \V_f^h}
    \frac{\prodOf{\div w^h}{p^h} - \prodGi{w^h \cdot \n}{p^h}} {\normOf{\nabla w^h}}.
\end{align*}
Now, \eqref{infsup} follows by means of \cref{lem_stab_gen}.
\end{proof}
\begin{remark}
  Note that, for a normalized pressure space $\P^h \subset L^2_0(\Omega_f)$,
  we could obtain an inf-sup condition for the $L^2$-norm $\|p^h\|_{\Omega_f}$
  from~\eqref{infsup} by means of a Poincaré inequality.
  In the case of FSI, however, the pressure is normalized by the
  interface condition $p^h \n = 2 \rho_f \nu_f \epsilon(v_f) \n - \sigma_s \n$, such that an additional normalization is not feasible.
  Typical constructions to show an inf-sup condition
  when outflow boundaries are present
  require that the boundary coincides with a mesh line,
  see, e.g., \cite[Lem.~2]{Zulehner2024},
  and fail in the case of unfitted boundaries.
\end{remark}

Now we can show an $L^2(H^1)$-stability result for the pressure
under the assumption of an inverse parabolic CFL condition.
\begin{lemma}%
  \label{lem_pressstab}%
  Let $h^2 \lesssim k$
  and let Assumption~\ref{ass:F} hold with $\theta > 0$.
  Then, any
  $U^{h,n} = (v_f^{h,n}, p^{h,n}, v_s^{h,n}, u^{h,n})$
  in $\X^h$, $n = 1,\dots, N$,
  {that satisfies \eqref{eq:fullydisc} also satisfies}
  \begin{align}
    \label{pressstab}
    \sum_{n=1}^N kh^2 \normOfh{\nabla p^{h,n}}^2
    \lesssim
    \CF^2 +
    \sum_{n=1}^N \bigl( \rho_f \normOf{\delta v_f^{h,n}}^2
    + k \normiii{U^{h,n}}^2 \bigr).
  \end{align}
\end{lemma}
\begin{proof}
  Due to \cref{lem_infsup}, there exists for each $p^{h,n} \in \P^h$ a function
  $w^{h,n} \in \V_f^h$ with $\normOfh{w^{h,n}} \le c h^2 \normOf{\nabla p^{h,n}}$ such that
  \begin{align}
    \label{infsup1}
    \gamma_1 h^2 \normOfh{\nabla p^{h,n}}^2
    &\le \prodOf{\div w^{h,n}}{p^{h,n}}
      - \prodGi{w^{h,n} \cdot \n}{p^{h,n}}
      + g_p^{h,w}(p^{h,n}, p^{h,n}).
  \end{align}
To estimate the first two terms
  on the right-hand side,
  we choose the test function $\Phi^{h,n} := (w^{h,n}, 0, 0, 0)$
  in \eqref{eq:fullydisc}
  and apply the Cauchy--Schwarz inequality repeatedly,
  \begin{equation}
    \label{langlang}
    \begin{aligned}
      \prodOf{\nabla &\cdot w^{h,n}}{p^{h,n}} - \prodGi{w^{h,n} \cdot \n}{p^{h,n}} \\
      &= \frac{\rho_f}{k} \prodOf{\delta v_f^{h,n}}{w^{h,n}}
      + \frac{\rho_f \nu_f \gamma_N}{h} \prodGi{v_f^{h,n} - v_s^{h,n}}{w^{h,n}}
      \\[-2pt]
      &\quad- 2 \rho_f \nu_f \prodGi{\symgrad{v_f^{h,n}} \n}{w^{h,n}}
      - 2 \rho_f \nu_f \prodGi{v_f^{h,n} - v_s^{h,n}}{\symgrad{w^{h,n}} \n}
      \\[-2pt]
      &\quad+ 2 \rho_f \nu_f \prodOf{\symgrad{v_f^{h,n}}}{\nabla w^{h,n}}
      + 2 \rho_f \nu_f g_{v_f}^{h,w}(v_f^{h,n}, w^{h,n})
      - F_f^n(w^{h,n}) \\
      &\le \frac{\rho_f}{k} \normOf{\delta v_f^{h,n}} \normOf{w^{h,n}}
      + 2 \rho_f \nu_f \normOf{\nabla v_f^{h,n}}\normOf{\nabla w^{h,n}} \\[-2pt]
      &\quad+ \frac{\sqrt{\rho_f\nu_f \gamma_N}}{h^{\frac12}}
      \normGi{v_f^{h,n}-v_s^{h,n}}
      \frac{\sqrt{\rho_f \nu_f \gamma_N}}{h^{\frac12}} \normGi{w^{h,n}} \\[-2pt]
      &\quad+ 2 \rho_f \nu_f \normGi{\nabla v_f^{h,n}} \normGi{w^{h,n}}
      + 2 \rho_f \nu_f
      h^{-\frac12} \normGi{v_f^{h,n} - v_s^{h,n}}
      h^{\frac12} \normGi{\nabla w^{h,n}} \\[-2pt]
      &\quad+ 2 \rho_f\nu_f
      g_{v_f}^{h,w}(v_f^{h,n}, v_f^{h,n})^{\frac12}
      g_{v_f}^{h,w}(w^{h,n}, w^{h,n})^{\frac12}
      - F_f^n(w^{h,n}).
    \end{aligned}
  \end{equation}
  Summing over $n = 1, \dots, N$ and using Assumption~\ref{ass:F}
  for $\Phi^{h,n} = (w^{h,n}, 0,0,0)$, we obtain for the last term,
  using an inverse inequality (see Lemma~\ref{traceineq}),
  \begin{align*}
    \sum_{n=1}^N k |F_f^n(w^{h,n})|
    &\le \CF \biggl[ \frac{\rho_f}{2} \normOf{w^{h,N}}^2
      + \sum_{n=1}^N k \bigl( \rho_f \nu_f \normOfh{\nabla w^{h,n}}^2
      + \rho_f \nu_f \gamma_N \tracenorm{w^{h,n}}^2 \bigr)
      \biggr]^{\!\frac12} \\
   &\lesssim \CF \biggl[ \frac{\rho_f}{2} \normOf{w^{h,N}}^2
     + \sum_{n=1}^N k \bigl( \normOfh{\nabla w^{h,n}}^2
     + h^{-2} \normOfh{w^{h,n}}^2 \bigr) \biggr]^{\!\frac12}.
  \end{align*}
  Next, we use the estimates~\eqref{whH1},  Lemma~\ref{traceineq}, as well as the relations
  \begin{align*}
    \normGi{\nabla^l w^{h,n}}
    &\le ch^{-\frac12}\normOfh{\nabla^l w^{h,n}}
      \le ch^{\frac32-l}\normOf{\nabla p^{h,n}} \quad\text{for } l\in\{0,1\}, \\
    g_{v_f}^{h,w}(w^{h,n}, w^{h,n})^{\frac12}
    &\le c\normOfh{\nabla w^{h,n}}
     {} \le c h \normOf{\nabla p^{h,n}}, \\
    g_{v_f}^{h,w}(v_f^{h,n}, v_f^{h,n})^{\frac12}
    &\le c\normOfh{\nabla v_f^{h,n}}.
  \end{align*}
With further inverse estimates,
  we obtain from \eqref{langlang}
with $\delta v_f^{h,n} := v_f^{h,n} -  v_f^{h,n-1}$:\begin{equation}
    \def\vdiff{\delta v_f^{h,n}}
    \label{long1}
    \begin{aligned}
      &\sum_{n=1}^N k \left( \prodOf{\div w^{h,n}}{p^{h,n}}
        - \prodGi{w^{h,n} \cdot \n}{p^{h,n}} \right) \\
      &\quad\lesssim \biggl[ h^4\normOf{\nabla p^{h,N}}^2
      + \sum_{n=1}^N k h^2 \normOf{\nabla p^{h,n}}^2 \biggr]^{\!\frac12} \\
      &\quad\times\biggl[ {\CF} + \biggl( \sum_{n=1}^N
      \rho_f \frac{h^2}{k} \normOf{\vdiff}^2
      + k \rho_f \nu_f \normOfh{\nabla v_f^{h,n}}^2
      + k \frac{\rho_f \nu_f \gamma_N}{h} \normGi{v_f^{h,n}-v_s^{h,n}}^2
      \biggr)^{\!\frac12} \biggr].
    \end{aligned}
  \end{equation}
  Summing \eqref{infsup1} over $n = 1, \dots, N$, multiplying with $k$ and using that $h^2 \lesssim k$,
  we can absorb the pressure terms in \eqref{long1}
  into the left-hand side of \eqref{infsup1}, which yields the statement.
\end{proof}

Now we can combine Theorem~\ref{theo_stab} and Lemma~\ref{lem_pressstab} to obtain a stability result for all variables.

\begin{theorem}%
  \label{theo.stabpress}%
  Under the condition $h^2\lesssim k$, {Assumption~\ref{ass:F}} and the assumptions of \cref{theo_stab},
  it holds for any $U^{h,n} = (v_f^{h,n}, p^{h,n}, v_s^{h,n}, u^{h,n})$ in $\X^h$,
  $n = 1, \dots, N$,
  that satisfy \eqref{eq:fullydisc}:
  \begin{equation}
    \label{est_en}
    {\begin{aligned}
      \ET(U^{h,N})^2
      &+ \sum_{n=1}^N \bigl( k \normiii{U^{h,n}}^2
        + kh^2 \normOfh{\nabla p^{h,n}}^2
        + \ET(\Udiff)^2 \bigr) \\
      &\lesssim \expcT \biggl[
        \ET(U^{h,0})^2 + \CF^2 +
        \sum_{n=1}^N k \norm[H^1(\Osh)]{\vres}^2 \biggr].
    \end{aligned}}
  \end{equation}
\end{theorem}

\begin{proof}
  {We start by applying Korn's second inequality to $u^{h,m}-u^{h,0}$ in combination with~\eqref{est:u} to obtain
  \begin{equation}
    \label{eq:terminal-korn}
    \normOsh{\nabla u^{h,m}}^2
    \lesssim
    \normOsh{\nabla u^{h,0}}^2
    +\normOs{\symgrad{u^{h,m}}}^2
    +g_u^{h,w}(u^{h,m},u^{h,m})
    +\normOs{u^{h,m}-u^{h,0}}^2.
  \end{equation}
  Using the definition of $\bm r_v^{h,n}$~\eqref{defvres}, we can reformulate the last term as
  \begin{equation}\label{uhmuh0}
    u^{h,m}-u^{h,0}=k\sum_{n=1}^m(v_s^{h,n}-\bm r_v^{h,n}).
  \end{equation}
  for every $m=1,\ldots,N$.
  Hence, by the Cauchy--Schwarz inequality,
  \begin{equation}
    \label{eq:kinematic-sum}
    \normOs{u^{h,m}-u^{h,0}}^2\leq2t_m\sum_{n=1}^m k
      \bigl(\normOs{v_s^{h,n}}^2+\normOs{\vres}^2\bigr).
  \end{equation}
  Subtracting~\eqref{uhmuh0} at consecutive time levels and testing
  the resulting identity with $2(u^{h,m}-u^{h,0})$ in $L^2(\Omega_s)$ yields
  \begin{align*}
    \normOs{u^{h,m}-u^{h,0}}^2
      -\normOs{u^{h,m-1}&-u^{h,0}}^2
    \leq2k\prodOs{v_s^{h,m}-\bm r_v^{h,m}}{u^{h,m}-u^{h,0}}\\
    &\leq k\normOs{u^{h,m}-u^{h,0}}^2
      +2k\bigl(\normOs{v_s^{h,m}}^2+\normOs{\bm r_v^{h,m}}^2\bigr).
  \end{align*}
  Summing gives
  \begin{equation}
    \label{eq:kinematic-gronwall}
      \normOs{u^{h,m}-u^{h,0}}^2
      \leq
      \sum_{n=1}^m k\normOs{u^{h,n}-u^{h,0}}^2\\
      +2\sum_{n=1}^m k
        \bigl(\normOs{v_s^{h,n}}^2+\normOs{\vres}^2\bigr).
  \end{equation}
  }{Next, we combine} \cref{theo_stab,lem_pressstab} { and insert
  \eqref{eq:fullydisc}. There exist constants $c_6,c_7>0$, independent
  of $h$ and $k$, such that}
  \begin{equation}
    \label{stat1}
    \begin{aligned}
      {c_6}\biggl[&{\frac{\rho_f}{2}\normOf{v_f^{h,m}}^2
        +\frac{\rho_s}{2}\normOs{v_s^{h,m}}^2
        +\frac{\rho_s}{2}g_{v_s}^{h,w}(v_s^{h,m},v_s^{h,m})}\\
      &{+\mu_s\normOs{\symgrad{u^{h,m}}}^2
        +\frac{\lambda_s}{2}\normOs{\tr{\nabla u^{h,m}}}^2
        +\mu_s g_u^{h,w}(u^{h,m},u^{h,m})}\\
      &+\sum_{n=1}^{{m}}\bigl(k\normiii{U^{h,n}}^2
        +kh^2\normOfh{\nabla p^{h,n}}^2
        +\ET({\delta U^{h,n}})^2\bigr)\biggr]\\
      \leq{}&\sum_{n=1}^{{m}} kF^n(U^{h,n})
        +{c_7}\bigl(\ET({U^{h,0}})^2+\CF^2\bigr)\\
      &+{c_7}\sum_{n=1}^{{m}} k\bigl({\ET(U^{h,n})^2}
        +\norm[H^1(\Osh)]{\vres}^2\bigr).
    \end{aligned}
  \end{equation}
  {Using} Assumption~\ref{ass:F} and Young's
  inequality, {we obtain, for every $\varepsilon>0$,}
  \begin{equation}
    \label{eq:f}
    \begin{aligned}
      \sum_{n=1}^{{m}} kF^n(U^{h,n})
      \leq{}&{C_\varepsilon}\CF^2+{\varepsilon}\biggl[\ET(U^{h,{m}})^2\\
      &+\sum_{n=1}^{{m}} k\bigl(\normiii{U^{h,n}}^2
        +h^2\normOfh{\nabla p^{h,n}}^2+{\ET(U^{h,n})^2}
        +\norm[H^1(\Osh)]{\vres}^2\bigr)\biggr].
    \end{aligned}
  \end{equation}
  {We add a fixed positive multiple of
  \eqref{eq:kinematic-gronwall} to \eqref{stat1}. By
  \eqref{est:vs} and \eqref{eq:terminal-korn}, we have
  \begin{equation*}
    \begin{aligned}
      \ET(U^{h,m})^2\lesssim{}&
        \frac{\rho_f}{2}\normOf{v_f^{h,m}}^2
        +\frac{\rho_s}{2}\normOs{v_s^{h,m}}^2
        +\frac{\rho_s}{2}g_{v_s}^{h,w}(v_s^{h,m},v_s^{h,m})\\
      &+\mu_s\normOs{\symgrad{u^{h,m}}}^2
        +\frac{\lambda_s}{2}\normOs{\tr{\nabla u^{h,m}}}^2
        +\mu_s g_u^{h,w}(u^{h,m},u^{h,m})\\
      &+\mu_s\normOs{u^{h,m}-u^{h,0}}^2
        +\mu_s\normOsh{\nabla u^{h,0}}^2.
    \end{aligned}
  \end{equation*}
  The last term arises from splitting $u^{h,m}=u^{h,0}
  +(u^{h,m}-u^{h,0})$ in \eqref{eq:terminal-korn} and satisfies
  $\mu_s\normOsh{\nabla u^{h,0}}^2\leq\ET(U^{h,0})^2$.
  Inserting \eqref{eq:f} and choosing
  $\varepsilon$ sufficiently small therefore gives}
  \begin{equation*}
    \begin{aligned}
      \ET&(U^{h,{m}})^2{+\normOs{u^{h,m}-u^{h,0}}^2}
      +\sum_{n=1}^{{m}}\bigl(k\normiii{U^{h,n}}^2
      +kh^2\normOfh{\nabla p^{h,n}}^2
      +\ET({\delta U^{h,n}})^2\bigr)\\
      &\lesssim\ET({U^{h,0}})^2+\CF^2
      +\sum_{n=1}^{{m}} k\norm[H^1(\Osh)]{\vres}^2
      +\sum_{n=1}^{{m}} k{\bigl(\ET(U^{h,n})^2
      +\normOs{u^{h,n}-u^{h,0}}^2\bigr)}.
    \end{aligned}
  \end{equation*}
  {For $k$ sufficiently small,} a discrete Grönwall inequality
  {now yields \eqref{est_en}.}
\end{proof}

\begin{remark}
  {If $u^{h,0}=0$, the conclusion of \cref{theo.stabpress}
  also holds when Assumption~\ref{ass:F} is required only for test
  sequences with $\psi^{h,0}=0$, since all test sequences used in
  the proof satisfy this condition.}
  In \cref{theo.stabpress} we have not assumed that $\vres = 0$ ,
  which holds for the solution
of \cref{prob:fullydisc}.
  Under this assumption, the last term on the right-hand side
  of \eqref{est_en} would vanish.
  {Moreover, the residual terms in \eqref{newesti}
  and \eqref{stabsolid} vanish, and hence a direct telescoping argument
  controls the terminal strain energy without applying Grönwall's lemma.
  In this case, the exponential factor $\expcT$ in \eqref{est_en}
  can be replaced by a constant with at most polynomial dependence
  on $T$.}
  We state \cref{theo.stabpress} in the more general form
  as this allows us to apply it directly
  in the following a priori error analysis.
\end{remark}

\section{A priori error estimates}
\label{sec_apriori}
In this section, we establish our final a priori error estimates. We start with the introduction of a projection operator for the solid velocity $v_s$, which will help us in the error estimates to get rid of a problematic interface term. Then, the overall argumentation is as usual: we start with a specific Galerkin orthogonality and
then derive consistency and interpolation error estimates. Together with the stability result of the previous section, we then obtain the final a priori error estimates in time and space.

For the error analysis, we assume that the subdomains $\Omega_f$ and $\Omega_s$
as well as the interface $\Gamma^i$ are sufficiently smooth
and that with $\H_i := H^{m_i+1}(\Omega_i)$ for $i \in \set{f, s}$
the solution $U = (v_f, p, v_s, u)$ of \cref{weak_form} satisfies
\begin{equation}
	\label{regularity_assumption}
	\begin{aligned}
	  v_f &\in L^\infty(I, \H_f) \cap H^1(I, \H_f) \cap H^2(I, L^2(\Omega_f)), &
	  p &\in {{C}}([0,T], H^{m_f}(\Omega_f)), \\
	  v_s &\in L^\infty(I, \H_s) \cap H^1(I, \H_s) \cap H^2(I, L^2(\Omega_s)), \\
	  u &\in L^\infty(I, \H_s) \cap H^1(I, \H_s) \cap {H^2(I,\H_s)},
	\end{aligned}
\end{equation}
where $m_s \ge 1$, $m_f \ge 2$ are the polynomial degrees
of the solid respectively fluid finite element spaces. Because each component of $U$ is at least $H^2$-regular in space, nodal interpolation of $U$ is well-defined.
It is easy to see that $U$ fulfills the weak formulation
\begin{equation}\label{cont_time_system}
  \rho_f \prodOf{\partial_t v_f}{\phi_f^h} + \rho_s \prodOs{\partial_t v_s}{\phi_s^h} + A^h(U, \Phi^h) = F^h(\Phi^h)
\end{equation}
for all $\Phi^h \in \X^h$ at time $t \in I$.
As the physical subdomains
are strictly contained in the computational subdomains,
we need to extend the components of the solution $U$
to $\Ofh$ and $\Osh$, respectively.
For any $w \in L^\infty(I, H^r(\Omega_i))$,
there exist extensions
$
  E_i\: H^{r}(\Omega_i) \to H^r(\Omega_i^\T)
$
with $E_i|_{\Omega_i} = \textup{id}_{\Omega_i}$
for $i \in \set{f, s}$ that satisfy
\begin{align}
  \label{extension_space_stability}
  \norm[H^r(\Omega_i^\T)]{E_i w} \le C \norm[H^{r}(\Omega_i)]{w}.
\end{align}
If
$w \in {L^\infty}(I, H^{r+1}(\Omega_i)) \cap {W^{1,\infty}}(I, H^r(\Omega_i))$,
there exist extensions such that
\begin{align}
  \partial_t E_iw &= E_i \partial_t w \quad\text{in } \Omega_i^\T, &
  \label{extension_time_stability}
  \norm[H^r(\Omega_i^\T)]{\partial_t E_i w}
  \le C \norm[H^r(\Omega_i)]{\partial_t w}.
\end{align}
A proof of \eqref{extension_space_stability}
is given in \cite[Chap.~VI, §3.1]{stein1970singular}.
For \eqref{extension_time_stability}
we refer to \cite[Lem.~3.2]{Lehrenfeld2018}
and \cite[Lem.~3.3]{Lehrenfeld2018}.
For ease of notation, we will usually omit the operator $E_i$
when it is clear that it is needed.

\subsection{Projection and Interpolation operators}
Following~\cite{BURMAN2014FSI},
we introduce a special projection operator for the solid velocity.
To this end, we partition the set of interface cells and their neighbors
\begin{align*}
  \T_\Gamma^h := \defset{T \in \T_s^h}{\cl T \cap \cl G^h \ne \0}
\end{align*}
 in such a way into non-overlapping patches $P_j, j=1,\ldots,M$ that $|P_j|={\cal O}(h^2)$ and for the partition $\Gamma_j^i:=\Gamma^i\cap P_j$ of $\Gamma^i$ it holds  $|\Gamma_j^i|={\cal O}(h)$ (see \cref{fig:lem.proj}). Now, we construct the projection in such a way that
\begin{align}
   \label{intproj}
  \int_{\Gamma_j^i} (v_s - \pi_h v_s)\cdot \n_f \ds = 0, \quad j = 1, \dots, M.
\end{align}

\begin{figure}[t]
\begin{minipage}[c]{.3\linewidth}
\centering
 \begin{tikzpicture}[font = \scriptsize, scale = 1.5]
    \draw[color=black, thick] (-1, -1) rectangle (1, 1);
    \draw[color=black, fill=red!10, thick] (0, 0) circle (0.75);
    \ifXcolor
    \fill[pattern=north east lines, pattern color=Fuchsia] (-0.6, -0.8) rectangle (0.6, 0.8);
    \fill[pattern=north east lines, pattern color=Fuchsia] (-0.8, -0.6) rectangle (0.8, 0.6);

   \draw[color=red!10, fill=red!10] (-0.2, -0.4) rectangle (0.2, 0.4);
   \draw[color=red!10, fill=red!10] (-0.4, -0.2) rectangle (0.4, 0.2);
   \draw[step=0.2,black!60,thin] (-1, -1) grid (1, 1);
   \draw[color=red!10, fill=red!10] (0, 0) circle (0.13);
   \draw (0, 0) node {\scriptsize \textcolor{red}{$\Omega_s$}};
   \draw[color=Bittersweet, thick] (-0.2, -0.8) rectangle (0.2, -0.4);
   \draw[color=Bittersweet, thick] (-0.2, 0.8) rectangle (0.2, 0.4);
   \draw[color=Bittersweet, thick] (0.4, -0.2) rectangle (0.8, 0.2);
   \draw[color=Bittersweet, thick] (-0.4, -0.2) rectangle (-0.8, 0.2);
   \foreach \x in {-1, 1}
      \foreach \y in {-1, 1}
      {
          \draw[color=Bittersweet, thick] (\x*0.2, \y*0.8) -- (\x*0.6, \y*0.8);
          \draw[color=Bittersweet, thick] (\x*0.6, \y*0.6) -- (\x*0.8, \y*0.6);
          \draw[color=Bittersweet, thick] (\x*0.6, \y*0.6) -- (\x*0.6, \y*0.8);
          \draw[color=Bittersweet, thick] (\x*0.8, \y*0.2) -- (\x*0.8, \y*0.6);
          \draw[color=Bittersweet, thick] (\x*0.2, \y*0.2) -- (\x*0.4, \y*0.2);
          \draw[color=Bittersweet, thick] (\x*0.2, \y*0.2) -- (\x*0.2, \y*0.4);
      }
  \draw[color=Bittersweet, fill=white, thick] (-0.95, 0.7) rectangle (-0.7, 0.95);
  \draw (-0.825, 0.825) node {\scriptsize \textcolor{Bittersweet}{$P_j$}};
  \draw[color=white, fill=white] (0.8, -0.8) circle (0.13);
  \draw (0.8, -0.8) node {\scriptsize \textcolor{Fuchsia}{$\T_\Gamma^h$}};
  \fi
 \end{tikzpicture}
\end{minipage}
\begin{minipage}[c]{.28\linewidth}
\centering
 \begin{tikzpicture}[font = \scriptsize, scale = 1.3]
    \draw[color=red!10, fill=red!10] (-0.1, -1.1) rectangle (1.43693177121688, 1.5);
    \draw[black, very thick, name path=A] plot [smooth, tension = 0.7] coordinates {(1.5850383540486704, -1.1) (1.6142080737002401, -1.0) (1.7165171868296264, -0.5) (1.75, 0.0) (1.7165171868296264, 0.5) (1.6142080737002401, 1.0) (1.5850383540486704, 1.1) (1.43693177121688, 1.5)};
    \draw[red!10, name path=B] (1.43693177121688, -1.1) -- (1.43693177121688, 1.5);
    \tikzfillbetween[of=A and B,on layer=main]{red!10}
    \draw[step=1.0,black!60,thin] (-0.1, -1.1) grid (2.1, 1.5);
    \ifXcolor
    \draw[color=Bittersweet, very thick] (0, -1) rectangle (2, 1);
    \draw[magenta, very thick, name path=A] plot [smooth, tension = 0.7] coordinates {(1.6142080737002401, -1.0) (1.7165171868296264, -0.5) (1.75, 0.0) (1.7165171868296264, 0.5) (1.6142080737002401, 1.0)};
    \draw (0.5, -0.5) node {\scriptsize \textcolor{Bittersweet}{$P_j$}};
    \draw (1.5, -0.5) node {\scriptsize \textcolor{magenta}{$\Gamma_j^i$}};
    \fill[black] (1.0, 0.0) circle (1.5pt);
    \draw[anchor=south west] (1.0, 0.0) node {\textcolor{black}{$x_j$}};
    \fi
    \draw (0.25, 1.25) node {\scriptsize \textcolor{red}{$\Omega_s$}};
    \draw[|-|] (2.25, -1) -- (2.25, 1);
    \draw (2.45, 0.0) node {\scriptsize $2h$};
 \end{tikzpicture}
\end{minipage}
\begin{minipage}[c]{.4\linewidth}
 \centering
 \includegraphics[scale=0.1]{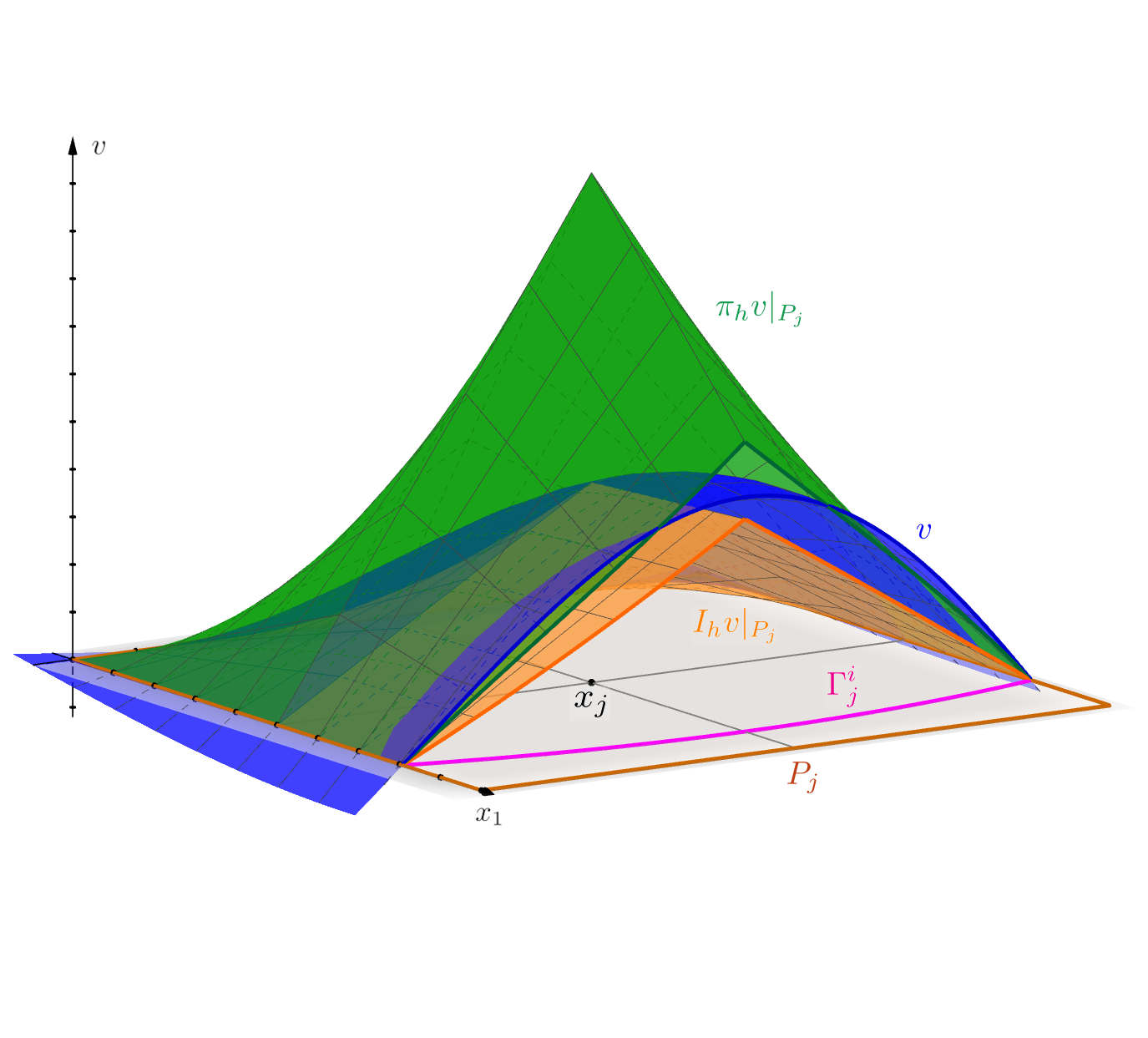}
\end{minipage}

\caption{Visualization of the set $\T_\Gamma^h$
  and an example of a partition into patches $P_1,\dots,P_M$ (left),
  a single patch $P_j$ with the interface part $\Gamma_j^i$ and a suitable vertex $x_j$ (middle)
  and a visualization of the projection $\pi_h v$ on a patch $P_j$
  for a scalar function $v$ cut off at the interface to make the differences visible (right) used in \cref{lem.proj}. Note that it holds $\int_{\Gamma_j^i} I_h v\, \text{d}s < \int_{\Gamma_j^i}\, v\, \text{d}s $, but $\int_{\Gamma_j^i} \pi_h v\, \text{d}s = \int_{\Gamma_j^i} v\, \text{d}s $.}
\label{fig:lem.proj}
\end{figure}

\begin{lemma}\label{lem.proj}
  Let $\T_\Gamma^h$ be partitioned into patches $P_j$,
  $j = 1, \dots, M$, as specified above.
  Let $h$ be sufficiently small and $\Gamma^i$ sufficiently smooth.
  Then, there exists a projection operator
  ${\pi_h\colon H^2(\Omega_s)^2\cap\V_s \to \V_s^h}$,
  such that~\eqref{intproj} holds and,
  for sufficiently regular functions $v$,
\begin{align}\label{proj1}
  \norm{v-\pi_h v}_{H^l({\Omega_s})}
  &\le ch^{r-l} \norm[H^{r}({\Omega_s})]{v}
  &&\text{for} \quad 0 \le l \le 1, \quad 2 \le r \,{{\le}\, m_s+1}, \\
  \norm[H^l(\bd{\Omega_s})]{v - \pi_h v}
  &\le ch^{r-l-\frac12} \norm[H^{r}({\Omega_s})]{v}
  &&\text{for} \quad 0 \le l \le  {1}, \quad 2 \le r \,{{\le}\, m_s+1}.\label{proj2}
  \end{align}
  \end{lemma}
  \begin{proof}
    The construction of $\pi_h$ is given in~\cite{BeckerBurmanHansbo2009}
    for the case of $P_1$ finite elements on triangular meshes and can
    easily be adapted to $Q_r$ finite elements on quadrilateral meshes.
    The idea is to set $\pi_h v|_{P_j} = I_h v|_{P_{j}} + y_j \phi_j$,
    where $I_h$ is the standard nodal interpolation,
    $\phi_j \in \mathcal{V}_s^h$ is a shape function belonging to an interior vertex $x_j$ of $P_j$
    such that $c_1 h \le \int_{\Gamma_j^i} \phi_j \ds \le c_2 h$
    with $c_1, c_2>0$ (see Fig.~\ref{fig:lem.proj} middle) and $y_j =\alpha_j \n_f(x_j) \in \R^2$, where $\n_f(x_j)$ is the outer fluid normal vector at $x_j$. The constant $\alpha_j$ is chosen as
  \begin{align*}
    \alpha_j=\frac{\int_{\Gamma_j^i} (v - I_h v)(s)\cdot \n_f(s) \ds}{\int_{\Gamma_j^i} \phi_j(s) \n_f(x_j)\cdot \n_f(s) \ds}.
  \end{align*}
Since $\Gamma^i$ is assumed to be smooth and $\|\n_f\|=1$, the term $\n_f(x_j)\cdot \n_f(s)$ in the denominator is bounded below by a constant $c\tsb{min}>0$ for a sufficiently fine mesh.

Eq.~\eqref{intproj} follows immediately from the definition of $\alpha_j$.
  The estimates~\eqref{proj1} and~\eqref{proj2}
  follow from the corresponding estimates for the nodal interpolant and
  standard estimates for the shape function $\phi_j$:
  \begin{align*}
  |\alpha_j| \leq c\tsb{min}^{-1}&\frac{\bigl|\int_{\Gamma_j^i} (v-I_h v)\cdot\n_f \ds \bigr|}{|\int_{\Gamma_j^i} \phi_j(s) \ds|} \lesssim h^{-1} \|v-I_h v\|_{L^1(\Gamma_j^i)} \lesssim h^{-1/2} \|v-I_h v\|_{L^2(\Gamma_j^i)} \\
  &\qquad\qquad\qquad\quad\lesssim h^{-1} \|v-I_h v\|_{P_j} + \|\nabla (v-I_h v)\|_{P_j} \leq ch^{r-1} \|v\|_{H^r(P_j)}.
  \end{align*}
  For further details, we refer to~\cite[Section 5.1]{BeckerBurmanHansbo2009}.
  \end{proof}
For the remaining variables, we use the standard Lagrangian interpolant $I_h$
to $\T_h^f$ respectively $\T_h^s$.
We introduce the notation
\begin{align*}
  U^n
  &:= {E U(t_n) := (E_f v_f(t_n), E_f p(t_n), E_s v_s(t_n), E_s u(t_n))}, \\
  \e_U^n
  &:= U^n - U^{h,n}, \quad
    \etab_U^n:= U^n - \I_h U^n, \quad
    \xib_{h,U}^n := \I_h U^n - U^{h,n}
\end{align*}
for ${n=0,\ldots,N}$, where $\I_h U := (I_h v_f, I_h p, \pi_h v_s, I_h u)$.
{At $n=0$, the error splitting is used only for the velocity
and displacement components; the discrete pressure component is left
unspecified. All expressions involving the initial state below are
independent of this component.
The initial data specified in Problem~\ref{prob:fullydisc} then give
\[
  \xib_{h,v_f}^0=0,\qquad
  \xib_{h,v_s}^0=(\pi_h-I_h)E_sv_s^0,\qquad
  \xib_{h,u}^0=0,
\]
Thus, only the solid-velocity component of the discrete initial error
is nonzero and approximation and extension estimates yield
\begin{align}\label{eq:initial-error}
  \ET(\xib_{h,U}^0)^2
  \lesssim h^{2(m_s+1)}\norm[\H_s]{v_s^0}^2.
\end{align}
}
{For the nodal interpolant $I_h$, we have the following
standard estimates for sufficiently regular displacements $u$:}
\begin{align}
  \label{interpol1}
  {\norm{u-I_h u}_{H^l(\Omega_s)}}
  &{\le ch^{r-l} \norm[H^r(\Omega_s)]{u}}
  &&\text{for} \quad 0 \le l \le 1, \quad {2 \le r \le m_s+1}, \\
  {\norm[H^l(\bd\Omega_s)]{u-I_h u}}
  &{\le ch^{r-l-\frac12} \norm[H^r(\Omega_s)]{u}}
  &&\text{for} \quad 0 \le l \le {1}, \quad {2 \le r \le m_s+1}.
\end{align}
{Analogous estimates hold for $v_f$ and $p$ on $\Omega_f$,
with the upper bound $m_s+1$ replaced by $m_f+1$ and $m_f$,
respectively.
For functions extended to $\Omega_i^{\mathcal T}$, the same estimates hold
on the corresponding computational domains.}

\subsection{Galerkin Orthogonality}

As a starting point for the error estimation,
we subtract \eqref{eq:fullydisc} from \eqref{cont_time_system}
to obtain for all $\Phi^h = (\phi_f^h, \xi^h, \phi_s^h, \psi^h)$ in $\X^h$\begin{equation}
  \label{Galerkin}
  \begin{aligned}
    \A^{kh}(\e_U^n, \e_U^{n-1}; \Phi^h)
    &= k \E_c^n(U^n, U^{n-1}; \Phi^h) ,
    \quad \text{where} \\
    \E_c^n(U^n, U^{n-1}; \Phi^h)
    &= \rho_f \prodOf{k^{-1} \delta v_f^n - \partial_t v_f(t_n)}{\phi_f^h}
    + S^h(U^n, \Phi^h) \\[-2pt]
    &+ \rho_s \prodOs{k^{-1} \delta v_s^n - \partial_t v_s(t_n)}{\phi_s^h}
    + k^{-1} \rho_s \ghw[v_s]{\delta v_s^n}{\phi_s^h}.
  \end{aligned}
\end{equation}
We further split \eqref{Galerkin} into
interpolation error and discrete error parts,
\begin{equation}
  \label{Galerkin2}
  \A^{kh}(\xib_{U,h}^n, \xib_{U,h}^{n-1}; \Phi^h)
  = k \E_c^n(U^n, U^{n-1}; \Phi^h) - k \E_i^n(U^n, U^{n-1}, \Phi^h)
  \quad \forall \Phi^h \in \X^h,
\end{equation}
where the interpolation error is defined as
\begin{align}
  \label{interpolError}
  \E_i^n(U^n, U^{n-1}, \Phi^h)
  := k^{-1} \A^{kh}(\etab_U^n, \etab_U^{n-1}; \Phi^h).
\end{align}
We will apply the stability result of \cref{theo.stabpress} to \eqref{Galerkin2},
which will be the basis of the error estimate.
Let us first estimate the consistency and interpolation errors.

\subsection{Consistency and interpolation errors}

Subsequently we will use the notation
$I_n := (t_{n-1}, t_n)$ for the $n$-th time interval.
{We let $m_{fs} := \min\set{m_f, m_s}$.}

\begin{lemma}[Consistency error]%
  \label{lem.consistency}%
  Let the solution $U$ to Problem~\ref{weak_form} satisfy the regularity assumption~\eqref{regularity_assumption} and let $U^n:=U(t_n)$ for $n=0,\ldots,N$.
Then, {for every $1\le m\le N$,} it holds for $\Phi^{h,n} = (\phi_f^{h,n}, \xi^{h,n}, \phi_s^{h,n}, \psi^{h,n}) \in \X^h, n = 1, \dots, {m}$:
  \begin{align*}
    \sum_{n=1}^{{m}} k
    &\E_c^n(U^n, U^{n-1}; \Phi^{h,n}) \\[-4.5\jot]
    &\le c k \left( \norm[L^2(I,\Omega_f)]{\partial_t^2 v_f}^2
      + \norm[L^2(I,\Omega_s)]{\partial_t^2 v_s}^2 \right)^{\frac12}
      \biggl( \sum_{n=1}^{{m}} k \bigl( \normOf{\phi_f^{h,n}}^2
      + \normOs{\phi_s^{h,n}}^2 \bigr) \biggr)^{\!\frac12}.
  \end{align*}
\end{lemma}

\begin{proof}
  For the first part of the consistency error,
  we have for $i \in \set{f, s}$
  \begin{equation}
    \label{consdtv}
    \begin{aligned}
      \delta v_i^n
      &- k \partial_t v_i(t_n)
      = \int_{I_n} [\partial_t v_i(t) - \partial_t v_i(t_n)] \dt {}= \int_{I_n} [\partial_t v_i(t) - \partial_t v_i(t_n)]\cdot 1 \dt.
          \end{aligned}
  \end{equation}
  We use that a primitive function to $1$ is given by $(t-t_{n-1})$ and apply integration by parts, followed by a Cauchy--Schwarz inequality in time:
  \begin{equation}
      \begin{aligned}\label{consdtv2}
  \int_{I_n} [\partial_t v_i(t) &- \partial_t v_i(t_n)]\cdot 1 \dt = -\int_{I_n} (t - t_{n-1}) \partial_t^2 v_i(t) \dt \\
      &\le \biggl( \int_{I_n} (t - t_{n-1})^2 \dt \biggr)^{\!\frac12}
      \biggl( \int_{I_n} (\partial_t^2 v_i(t))^2 \dt \biggr)^{\!\frac12}
      =\frac{1}{\sqrt{3}} k^{\frac32}
      \biggl( \int_{I_n} (\partial_t^2 v_i(t))^2 \dt \biggr)^{\!\frac12}.
    \end{aligned}
    \end{equation}
  Using a Cauchy-Schwarz inequality in space, this implies
  \begin{equation}
    \label{consTime1}
    \begin{aligned}
      \biggl| \rho_i
      &\sum_{n=1}^{{m}} \bigl( \iprod[\Omega_i]{\delta v_i^n}{\phi_i^{h,n}}
      - k \iprod[\Omega_i]{\partial_t v_i(t_n)}{\phi_i^{h,n}} \bigr) \biggr|
      \\[-\jot]
      &\le c k^{\frac32} \sum_{n=1}^{{m}}
      \norm[L^2(I_n,\Omega_i)]{\partial_t^2 v_i} \normOi{\phi_i^{h,n}}
      \le c k \norm[L^2(I,\Omega_i)]{\partial_t^2 v_i}
      \biggl( \sum_{n=1}^{{m}} k \normOi{\phi_i^{h,n}}^2 \biggr)^{\!\frac12}.
    \end{aligned}
  \end{equation}
  Due to the regularity assumption~\eqref{regularity_assumption}, we have $v_f^n\in H^{m_f+1}(\Ofh)$, $p^n\in H^{m_f}(\Ofh)$ and $ \delta v_s^n, u^n\in H^{m_s+1}(\Osh)$, such that the jumps of derivatives within the ghost penalties vanish. Hence, we have $S^h(U^n, \Phi^h)=k^{-1}\rho_sg_{v_s}^{h,w}(\delta v_s^n, \phi_s^h)=0$, which concludes the proof.
\end{proof}

In order to formulate an estimate for the interpolation error, let $\bm U := (U^n)_{n=1}^N = (v_f^n, p^n, v_s^n, u^n)_{n=1}^N$
and introduce the semi-norm
\begin{align*}
  |\bm U|_m^2
  &:= \norm[L^2(I,L^2(\Omega_f))]{\partial_t^2 v_f}^2
  + \norm[L^2(I,L^2(\Omega_s))]{\partial_t^2 v_s}^2
  + {\norm[L^2(I,\H_s)]{\partial_t^2 u}^2}
  + \norm[\H_s]{u^N}^2
  \\[-1pt]
  &\qquad{{}+\norm[\H_s]{v_s^0}^2}
    + \norm[L^2(I,\H_f)]{\partial_t v_f}^2
    + \norm[L^2(I,\H_s)]{\partial_t v_s}^2
    + \norm[L^2(I,\H_s)]{\partial_t u}^2
  \\[-2pt]
  &\qquad+ k \sum_{n=1}^N \left( \norm[\H_f]{v_f^n}^2
    + \norm[H^{m_f}(\Omega_f)]{p^n}^2
    + \norm[\H_s]{v_s^n}^2
    + \norm[\H_s]{u^n}^2 + { \norm[\H_s]{\partial_t u(t_n)}^2 }\right).
\end{align*}

\begin{lemma}[Interpolation error]%
  \label{lem.interpolation}%
  Let $\Phi^{h,n} = (\phi_f^{h,n}, \xi^{h,n}, \phi_s^{h,n}, \psi^{h,n})$ in $\X^h$
  with $\psi^0 = 0$ and
  $\pres = \phi_s^{h,n} - k^{-1} (\psi^{h,n} - \psi^{h,n-1})$.
  Then, {for every $1\le m\le N$,} it holds
  \begin{align*}
    \sum_{n=1}^{{m}} k \E_i^n(&U^n, U^{n-1}, \Phi^{h,n}) \le {}
    \ c {(1+T)^{\frac12}} h^{{m_{fs}}} |\bm U|_m
      \biggl[ \ET(\Phi^{h,{m}})^2 +\\[-2\jot]
       &\sum_{n=1}^{{m}} k \bigl(
      {\normiii{\Phi^{h,n}}^2}%
      + h^2\normOf{\nabla\xi^{h,n}}^2
      + \normOsh{\phi_s^{h,n}}^2 + \normOsh{\nabla \psi^{h,n}}^2
      + \normOsh{\nabla \pres}^2 \bigr)
      \biggr]^{\!\frac12}.
  \end{align*}
  {Here, $c$ is independent of $T$, $h$, $k$ and $m$.}
\end{lemma}

\begin{proof}
  We estimate the interpolation error \eqref{interpolError} term by term.
  {For the first term, write $J_{f,h}:=I_h$ and $J_{s,h}:=\pi_h$.
  These time-independent linear operators commute with temporal
  differentiation. We have}
  \begin{equation}
    \label{interpoldteta}
    \begin{aligned}
      \Bigl| \frac1k
      &\iprod[\Omega_i]{\etab_{v_i}^n - \etab_{v_i}^{n-1}}{\phi_i^{h,n}} \Bigr| \\
      &\le \frac1k
      \normOi{\etab_{v_i}^n - \etab_{v_i}^{n-1}} \normOi{\phi_i^{h,n}}
      = \frac1k \Norm[\Omega_i]{\int_{I_n}
        \partial_t (v_i(t) - {J_{i,h}} v_i(t)) \times 1 \dt} \normOi{\phi_i^{h,n}} \\
      &\le k^{-\frac12}
      \norm[L^2(I_n,\Omega_i)]{\partial_t v_i - {J_{i,h}}\partial_t v_i}
      \normOi{\phi_i^{h,n}}
      \le h^{m_i+1} k^{-\frac12}
      \norm[L^2(I_n,\H_i)]{\partial_t v_i}
      \normOi{\phi_i^{h,n}}.
    \end{aligned}
  \end{equation}
  Summing over $n=1,\dots,{m}$ and using the Cauchy--Schwarz inequality, we obtain
  \begin{align*}
    \sum_{n=1}^{{m}}k \Bigl| \frac1k
    \iprod[\Omega_i]{\etab_{v_i}^n - \etab_{v_i}^{n-1}}{\phi_i^{h,n}} \Bigr|
    &\le \sum_{n=1}^{{m}}  h^{m_i+1}
      \norm[L^2(I_n,\H_i)]{\partial_t v_i} k^{\frac12}
      \normOi{\phi_i^{h,n}} \\[-\jot]
    &\lesssim h^{m_i+1} \norm[L^2(I,\H_i)]{\partial_t v_i}
      \biggl( \sum_{n=1}^{{m}} k \normOi{\phi_i^{h,n}}^2 \biggr)^{\!\frac12}.
  \end{align*}
  For the fluid bulk terms, we have by means of standard estimates
  \begin{align*}
    \prodOf{\sigma_f^h(\etab_{v_f}^n,\etab_p^n)}{\nabla \phi_f^{h,n}}
    &\lesssim h^{m_f}
      \left( \normOf{\nabla^{m_f+1} v_f^n} + \normOf{\nabla^{m_f} p^n} \right)
      \normOf{\nabla \phi_f^{h,n}},
  \end{align*}
  and using integration by parts
  \begin{align}
\begin{split}
  \label{firstnf}
\prodOf{\div \etab_{v_f}^n}{\xi^{h,n}}
&= -\prodOf{\etab_{v_f}^n}{\nabla\xi^{h,n}} + (\etab_{v_f}^n,\xi^{h,n}\n_f)_{\Gamma_i}\\
&{\le C} h^{m_f}\normOf{\nabla^{m_f+1} v_f^n}h\normOf{\nabla\xi^{h,n}} +  (\etab_{v_f}^n,\xi^{h,n}\n_f)_{\Gamma_i}.
\end{split}
\end{align}
  The Nitsche penalty term can be estimated as follows:
  \begin{align*}
    \frac{\rho_f\nu_f\gamma_N}{h}
    &\prodGi{\etab_{v_f}^n - \etab_{v_s}^n}{\phi_f^{h,n} - \phi_s^{h,n}}
      \le \frac{\rho_f\nu_f\gamma_N}{h^{\frac12}}
      \bigl( \normGi{\etab_{v_f}^n} + \normGi{\etab_{v_s}^n} \bigr)
      \tracenorm{\phi_f^{h,n} - \phi_s^{h,n}} \\
    &\ \lesssim \rho_f \nu_f \gamma_N
      \left( h^{m_f} \normOf{\nabla^{m_f+1} v_f^n }
      + h^{m_s} \normOs{\nabla^{m_s+1} v_s^n} \right)
      \tracenorm{\phi_f^{h,n} - \phi_s^{h,n}}.
  \end{align*}
  Similarly, we get for the interface stresses
  \begin{align*}
    -\prodGi{\sigma_f^h (\etab_{v_f}^n&, \etab_p^n) \n_f}
    {\phi_f^{h,n} - \phi_s^{h,n}} \\
    &\lesssim h^{m_f} \left( \normOf{\nabla^{m_f+1} v_f^n}
      + \normOf{\nabla^{m_f} p^n} \right)
      \tracenorm{\phi_f^{h,n} - \phi_s^{h,n}}.
      \end{align*}
The third interface term can be estimated as follows
      \begin{align*}
    -\prodGi{\etab_{v_f}^n - \etab_{v_s}^n}
    {\sigma_f^h(\phi_f^{h,n}, &-\xi^{h,n}) \n_f}
    {\le} -\prodGi{\etab_{v_f}^n}
    {\xi^{h,n} \n_f} + \prodGi{\etab_{v_s}^n}{\xi^{h,n} \n_f}\\
    &+{C}\left( h^{m_f}  \normOf{\nabla^{m_f+1} v_f^n}
      + h^{m_s} \normOs{\nabla^{m_s+1} v_s^n} \right)
     {\normOfh{\nabla \phi_f^{h,n}}}.
  \end{align*}
  The first term on the right cancels with the last term in~\eqref{firstnf}. For the last term of the first line, we use that, due to~\eqref{intproj}, we can insert a piecewise constant function $P_0 \xi^{h,n}$ (with respect to the partition $\Gamma_j^i$ of $\Gamma^i$) defined by the mean values
\begin{align*}
P_0\xi^{h,n}|_{\Gamma_j^i} :=|\Gamma_j^{i}|^{-1} \int_{\Gamma_j^i} \xi^{h,n} \ds.
\end{align*}
By standard estimates it holds that
\begin{align*}
\| \xi^{h,n} -P_0\xi^{h,n}\|_{\Gamma_j^i} \lesssim  h\|\nabla \xi^{h,n}\|_{\Gamma_j^i} \lesssim h^{1/2} \|\nabla \xi^{h,n}\|_{P_j},
\end{align*}
and hence
\begin{align*}
\prodGi{\etab_{v_s}^n}{\xi^{h,n} \n_f} &= \prodGi{\etab_{v_s}^n}{(\xi^{h,n}-P_0 \xi^{h,n}) \n_f}\lesssim h^{1/2}\|\etab_{v_s}^n\|_{\Gamma^i} \|\nabla \xi^{h,n}\|_{\Ofh} \\
&\lesssim h^{m_s} \normOs{\nabla^{m_s+1} v_s^n} \left(h\|\nabla \xi^{h,n}\|_{\Omega_f}+g_p^{h,w}(\xi^{h,n},\xi^{h,n})^{1/2}\right).
\end{align*}
  For the structure bulk terms we have,
  using the symmetry and linearity of $\sigma_s$,
  \begin{equation}
    \label{split42}
    \begin{aligned}
      \sum_{n=1}^{{m}}
      &k \prodOs{\sigma_s(\etab_u^n)}{\nabla \phi_s^{h,n}}
      = \sum_{n=1}^{{m}} k \prodOs{\sigma_s(\etab_u^n)}{\symgrad{\phi_s^{h,n}}} \\[-\jot]
      &= \sum_{n=1}^{{m}}
      \prodOs{\sigma_s(\etab_u^n)}{\symgrad{\delta \psi^{h,n}}}
      + k \prodOs{\sigma_s(\etab_u^n)}
      {\symgrad{\phi_s^{h,n} - k^{-1} \delta \psi^{h,n}}}.
    \end{aligned}
  \end{equation}
  The last part can be estimated by the Cauchy--Schwarz inequality
  and an interpolation estimate.
  Reorganizing the sum\ and using that $\psi^0 = 0$,
  we obtain for the first term \begin{align}
    \label{phipsi}
    \sum_{n=1}^{{m}} \prodOs{\sigma_s(\etab_u^n)}{\symgrad{\delta \psi^{h,n}}}
    = \prodOs{\sigma_s(\etab_u^{{m}})}{\symgrad{\psi^{h,{m}}}}
    - \sum_{n=1}^{{m}}
    \prodOs{\sigma_s(\etab_u^n - \etab_u^{n-1})}{\symgrad{\psi^{h,n-1}}}.
  \end{align}
  For the first term on the right-hand side,
  we have using interpolation estimates:
  \begin{align*}
    \prodOs{\sigma_s(\etab_u^{{m}})}{\symgrad{\psi^{h,{m}}}}
    \le c h^{m_s} \norm[\H_s]{u^{{m}}} \normOs{\symgrad{\psi^{h,{m}}}}.
  \end{align*}
  For the sum on the right-hand side of~\eqref{phipsi}, we proceed similarly to~\eqref{interpoldteta} by replacing the discrete time derivative of $\nabla \etab_u$ with a Bochner norm of the continuous time derivative $\partial_t u$:
  \begin{align*}
    \sum_{n=1}^{{m}}
    \prodOs{\sigma_s(\etab_u^n - \etab_u^{n-1})}{\symgrad{\psi^{h,n-1}}}
    &\lesssim \sum_{n=1}^{{m}} \normOs{\nabla \etab_u^n- \nabla \etab_u^{n-1}}
    \normOs{\nabla \psi^{h,n-1}} \\
    &\lesssim \sum_{n=1}^{{m}} \Norm[\Omega_s]{\int_{I_n}
      \partial_t {{\nabla} (u(t) - I_h u(t))} \dt} %
    \!\normOs{\nabla \psi^{h,n-1}} \\
    &\lesssim \sum_{n=1}^{{m}} h^{m_s}
    \norm[L^2(I_n,\H_s)]{\partial_t u}
    k^{\frac12} \normOs{\nabla\psi^{h,n-1}} \\
    &\lesssim h^{m_s} \norm[L^2(I,\H_s)]{\partial_t u}
    \biggl( \sum_{n=1}^{{m}} k \normOs{\nabla\psi^{h,n-1}}^2 \biggr)^{\!\frac12}.
  \end{align*}
  In the last step, we have used the Cauchy--Schwarz inequality.
  Similarly, we estimate the ghost penalty term:
  \begin{align*}
    \ghw[v_s]{\etab_{v_s}^n - \etab_{v_s}^{n-1}}{\phi_s^{h,n}}
    &\lesssim \wmax \sumFs \sum_{j=1}^{m_s} \frac{h^{2j+1}}{(j!)^2 }
    \iprod[F]{\jump{\partial_n^j (\etab_{v_s}^n - \etab_{v_s}^{n-1})}}
          {\jump{\partial_n^j\phi_s^{h,n}}} \\
    &\lesssim \wmax \sum_{K\in\Osh} \sum_{j=1}^{m_s} \frac{h^{2j+1}}{(j!)^2}
     \Norm[\bd K]{\int_{I_n} |\partial_t\partial_n^j\etab_{v_s}| \dt}
     \norm[\bd K]{\partial_n^j \phi_s^{h,n}} \\
    &\lesssim k^{\frac12} \wmax \sum_{K\in\Osh}
      \sum_{j=1}^{m_s} \frac{h^{2j+1}}{(j!)^2}
     \norm[L^2(I_n,L^2(\bd K))]{\nabla^j \partial_t \etab_{v_s}}
     \norm[\bd K]{\nabla^j \phi_s^{h,n}}\\
     &\lesssim k^{\frac12} \wmax \sum_{K\in\Osh}
     \biggl( \sum_{j=1}^{m_s} h^{j+\frac12}
     \norm[L^2(I_n,L^2(\bd K))]{\nabla^j \partial_t \etab_{v_s}}
     \biggr) \norm[K]{\phi_s^{h,n}}\\
    &\lesssim k^{\frac12} \wmax h^{m_s+1}
    \norm[L^2(I_n,\H_s)]{\partial_t v_s}
    \normOsh{\phi_s^{h,n}}.
  \end{align*}
  {Here, the interpolation estimates on the facets follow
  from the elementwise trace inequality and local interpolation estimates;
  see e.g.~\cite[Section~5.3 and Lemma~7.2]{MassingSchottWall2017}.
  For the modified projection $\pi_h$, we split
  $v_s-\pi_h v_s=(v_s-I_h v_s)+(I_h v_s-\pi_h v_s)$ and estimate the
  second term by inverse trace inequalities and Lemma~\ref{lem.proj}.

  Summation over $n=1,\ldots,m$ results in
  }
  \begin{align*}
    \sum_{n=1}^{{m}} \ghw[v_s]{\etab_{v_s}^n - \etab_{v_s}^{n-1}}{\phi_s^{h,n}}
    \lesssim h^{m_s+1} \norm[L^2(I,\H_s)]{\partial_t v_s}
    \biggl( \sum_{n=1}^{{m}} k \normOsh{\phi_s^{h,n}}^2 \biggr)^{\!\frac12}.
  \end{align*}
  The next ghost penalty term is split into
  \begin{align}\label{ghostsplit}
    k \sum_{n=1}^{{m}}
    \ghw[u]{\etab_u^n}{\phi_s^{h,n}}
    &= \sum_{n=1}^{{m}}
    \left( \ghw[u]{\etab_u^n}{\delta \psi^{h,n}} + k
    \ghw[u]{\etab_u^n}{\phi_s^{h,n}- k^{-1} \delta \psi^{h,n}} \right).
  \end{align}
  For the first part on the right-hand side, we obtain as in \eqref{phipsi}
  \begin{align*}
    \sum_{n=1}^{{m}} &\ghw[u]{\etab_u^n}{\delta \psi^{h,n}}
    = \ghw[u]{\etab_u^{{m}}}{\psi^{h,{m}}}
    - \sum_{n=1}^{{m}} \ghw[u]{\etab_u^n - \etab_u^{n-1}}{\psi^{h,n-1}}
    \\
    &\lesssim \wmax \biggl[
      h^{m_s} \norm[\H_s]{u^{{m}}} \normOsh{\nabla\psi^{h,{m}}}
      + h^{m_s} \norm[L^2(I,\H_s)]{\partial_t u}
      \biggl( \sum_{n=1}^{{m}} k \normOsh{\nabla \psi^{h,n-1}}^2 \biggr)^{\!\frac12}
      \biggr].
  \end{align*}
  {The endpoint terms are bounded uniformly in $m$ using
  $u^m=u^N-\int_{t_m}^T\partial_t u(t)\dt$, which gives
  \begin{align*}
    \norm[\H_s]{u^m}^2
    \le 2\norm[\H_s]{u^N}^2
      +2T\norm[L^2(I,\H_s)]{\partial_t u}^2
      \le 2(1+T)|\bm U|_m^2.
  \end{align*}}
  For the second part on the right-hand side of~\eqref{ghostsplit},
  the Cauchy--Schwarz inequality and an interpolation estimate result in
  \begin{align*}
    k \sum_{n=1}^{{m}}
    &\ghw[u]{\etab_u^n}{\phi_s^{h,n} - k^{-1} \delta \psi^{h,n})}
    \lesssim k \sum_{n=1}^{{m}} h^{m_s} \norm[\H_s]{u^n} \normOsh{\nabla \pres}.
  \end{align*}
  Finally, we have for the fluid ghost penalty terms
  by means of standard estimates
  \begin{align*}
   \ghw[v_f]{\etab_{v_f}^n}{\phi_f^{h,n}}
    &\lesssim h^{m_f} \normOf{\nabla^{m_f+1} v_f^n} \normOfh{\nabla\phi_f^{h,n}},\\
     \ghw[p]{\etab_{p}^n}{\xi^{h,n}}
    &\lesssim h^{m_f} \normOf{\nabla^{m_f} p^n} \ghw[p]{\xi^{h,n}}{\xi^{h,n}}^{\frac12}.
  \end{align*}
  {Bounding the sums of solution norms by the corresponding terms in
  $|\bm U|_m^2$ and collecting the estimates by the Cauchy--Schwarz
  inequality yields the claim.}
\end{proof}

\subsection{Final a priori error estimate}

Now, we are ready to prove the final a priori error estimate.

\begin{theorem}%
  \label{theo.energyerror}%
  Let $U^{h, n}$ be the solution
  of \cref{prob:fullydisc} for $m_s \ge 1$, $m_f \ge 2$.
  Let $U(t_n)$ be the continuous solution of \cref{strong_form}
  satisfying the regularity conditions \eqref{regularity_assumption}.
  Further, let $h^2 \lesssim k$ and let the assumptions of
  \cref{theo_stab} be satisfied.
  Then, it holds for $\e_U^n, n = 1, \dots, N$:
  \begin{equation}
    \label{energy_est}
    \begin{aligned}
      \ET(\e_U^N)^2 + \sum_{n=1}^N \bigl( k \normiii{\e_U^n}^2
      &+ kh^2 \normOfh{\nabla\e_p^n}^2
      + \ET(\e_U^n - \e_U^{n-1})^2 \bigr) \\[-2\jot]
      &\lesssim \expcT {\bigl( k^2 + h^{2m_{fs}} \bigr) |\bm U|_m^2}.
    \end{aligned}
  \end{equation}
  {Here, $\e_U^0=U^0-U^{h,0}$ denotes the initial velocity and
  displacement errors.}
\end{theorem}

\begin{proof}
  We divide $\e_U^n$ into
  a discrete part $\xib_{U,h}^n$ and an interpolation part $\etab_U^n$,
  \begin{align*}
    \ET(\e_U^N)^2
    &+ \sum_{n=1}^N \bigl( k \normiii{\e_U^n}^2
    + k h^2 \normOfh{\nabla\e_p^n}^2 + \ET(\e_U^n - \e_U^{n-1})^2 \bigr) \\
    {} {{\lesssim}\,} \ET(\xib_{U,h}^N)^2
    &+ \sum_{n=1}^N \bigl( k  \normiii{\xib_{U,h}^n}^2
    + k h^2 \normOfh{\nabla\xib_{p,h}^n}^2
    + \ET(\xib_{U,h}^n - \xib_{U,h}^{n-1})^2 \bigr) \\[-2pt]
    {} + \ET(\etab_U^N)^2
    &+ \sum_{n=1}^N \bigl( k \normiii{\etab_U^n}^2
    + kh^2 \normOfh{\nabla\etab_p^n}^2
    + \ET(\etab_U^n - \etab_U^{n-1})^2 \bigr).
  \end{align*}
  For the interpolation part, we have using~\eqref{interpol1}
  {,~\eqref{proj1}}
  and the stability of the extension
  \begin{align*}
    \ET(\etab_U^N)^2
    &\lesssim \normOf{\etab_{v_f}^N}^2
    + \normOsh{\etab_{v_s}^N}^2
    + \normOsh{\nabla \etab_u^N}^2 \\
    &\lesssim h^{2(m_f+1)} \norm[\H_f]{v_f^N}^2
    + h^{2(m_s+1)} \norm[\H_s]{v_s^N}^2
    + h^{2m_s} \norm[\H_s]{u^N}^2,
  \end{align*}
  and similarly to the proof of \cref{lem.interpolation} we find
  \begin{align*}
    \normiii{\etab_U^n}^2
    &+ h^2\normOfh{\nabla\etab_p^n}^2 \\
    &\lesssim \normOfh{\nabla\etab_{v_f}^n}^2
    + \tracenorm{\etab_{v_f}^n - \etab_{v_s}^{n}}^2
    + h^2\normOfh{\nabla\etab_p^n}^2
    +  {{\sumFf} \sum_{l=1}^{m_f-1} h^{2l+1} \norm[F]{\nabla^{l} \etab_p^n}^2} \\
    &\lesssim h^{2m_f} \norm[\H_f]{v_f^n}^2
    + h^{2m_s} \norm[\H_s]{v_s^n}^2
    + h^{2m_f} \norm[H^{m_f}(\Omega_f)]{p^n}^2.
  \end{align*}
  {Together, using the linearity of $I_h$ and $\pi_h$, we arrive at}
  \begin{equation}
    \label{eq:norms_interpol}
    \begin{aligned}
      \ET&(\etab_{U}^N)^2
      + \sum_{n=1}^N \bigl( k \normiii{\etab_U^n}^2
      + kh^2 \normOfh{\nabla\etab_p^n}^2
      + \ET(\etab_U^n - \etab_U^{n-1})^2 \bigr) \\
      &\lesssim h^{2(m_f+1)}
      \norm[\H_f]{v_f^N}^2
      + h^{2(m_s+1)}\norm[\H_s]{v_s^N}^2
      + h^{2m_s}\norm[\H_s]{u^N}^2 \\
      &+ \sum_{n=1}^N \left( k h^{2 m_f}\norm[\H_f]{v_f^n}^2
      + k h^{2m_s}\norm[\H_s]{v_s^n}^2
      + k h^{2m_f}\norm[H^{m_f}(\Omega_f)]{p^n}^2 \right. \\[-\jot]
      &\kern4em+  \left. h^{2(m_f+1)} \norm[\H_f]{\delta v_f^n}^2
      + h^{2(m_s+1)}\norm[\H_s]{\delta v_s^n}^2
      + h^{2m_s} \norm[\H_s]{\delta u^n}^2 \right).
    \end{aligned}
  \end{equation}
  Under the assumption $h^2\lesssim k$, the first two terms in the second line
  of \cref{eq:norms_interpol} can be absorbed into the third line.
  Additionally, the terms in the last line can be estimated
  similarly to \cref{interpoldteta}
  by replacing the subtraction with a temporal integral
  resulting in $L^2$-Bochner norms of the respective temporal derivatives.
  Together, we obtain
  \begin{equation*}%
    \begin{aligned}
      \ET(&\etab_{U}^N)^2
      + \sum_{n=1}^N \bigl( k \normiii{\etab_U^n}^2
      + k h^2\normOfh{\nabla\etab_p^n}^2
      + \ET(\etab_U^n - \etab_U^{n-1})^2 \bigr) \\
      &\lesssim h^{2(m_f+1)} k \norm[L^2(I,\H_f)]{\partial_t v_f}^2
      + h^{2{(m_s+1)}} k \norm[L^2(I, \H_s)]{\partial_t v_s}^2
      + h^{2{m_s}} k \norm[L^2(I,\H_s)]{\partial_t u}^2 \\[-2pt]
      &\quad+ h^{2m_s} \norm[\H_s]{u^N}^2
      + \sum_{n=1}^N \bigl( k h^{2m_f}\norm[\H_f]{v_f^n}^2
      + k h^{2m_s} \norm[\H_s]{v_s^n}^2
      + k h^{2m_f} \norm[H^{m_f}(\Omega_f)]{p^n}^2 \bigr).
    \end{aligned}
  \end{equation*}
  The discrete part $\xib_{U,h}^n$ satisfies
  for all $\Phi^{h,n} = (\phi_f^{h,n}, \xi^{h,n}, \phi_s^{h,n}, \psi^{h,n})$ in $\X^h$,
  $n = 1, \dots, N$, cf.~\eqref{Galerkin2}:
  \begin{equation}
    \label{Galerkin3}
    \sum_{n=1}^N \A^{kh}(\xib_{U,h}^n, \xib_{U,h}^{n-1}; \Phi^{h,n})
    = \sum_{n=1}^N k
    \left( \E_c^n(U^n, U^{n-1}; \Phi^{h,n}) - \E_i^n(U^n, U^{n-1}, \Phi^{h,n}) \right).
  \end{equation}
  We apply the stability result of \cref{theo.stabpress} to \eqref{Galerkin3}.
  Using \cref{lem.consistency,lem.interpolation},
  the right-hand side of \eqref{Galerkin3}{, summed up to any $1\le m\le N$,} is bounded by
  \begin{align*}
    \sum_{n=1}^{{m}} k
    \big( \E_c^n(&U^n, U^{n-1}; \Phi^{h,n})
    - \E_i^n(U^n, U^{n-1}, \Phi^{h,n}) \big) \lesssim \CF \biggl[ \ET(\Phi^{h,{m}})^2 +\\[-2\jot]
     &\sum_{n=1}^{{m}} k \bigl(
      \normiii{\Phi^{h,n}}^2%
      + h^2\normOf{\nabla\xi^{h,n}}^2 + \normOsh{\phi_s^{h,n}}^2 + \normOsh{\nabla\psi^{h,n}}^2 + \normOsh{\nabla\pres}^2 \bigr)
      \biggr]^{\!\frac12},
  \end{align*}
  where $\CF := ({(1+T)^{\frac12}}h^{{m_{fs}}} + k) |\bm U|_m$.
  {Thus, the right-hand side of~\eqref{Galerkin3} satisfies
  Assumption~\ref{ass:F} for test sequences with $\psi^{h,0}=0$, which is fulfilled for the discrete error, since $\xib_{u,h}^0=0$. Hence,}
    \cref{theo.stabpress} yields
  \begin{equation}
    \label{eq:discrete-error-stability}
    \begin{aligned}
    \ET&(\xib_{U,h}^N)^2
    + \sum_{n=1}^N \bigl( k \normiii{\xib_{U,h}^n}^2
    + kh^2 \normOfh{\nabla\xib_{p,h}^n}^2
    + \ET(\xib_{U,h}^n - \xib_{U,h}^{n-1})^2 \bigr) \\
    &\lesssim \expcT \biggl[ \ET(\xib_{U,h}^0)^2
    + \CF^2 + \sum_{n=1}^N k
    \norm[H^1(\Osh)]{\xib_{v_s}^{h,n} -
      k^{-1} ({{\xib}_{u,h}^n - \xib_{u,h}^{n-1}})}^2 \biggr] \\
    &{{}\lesssim \expcT \biggl[ \ET(\xib_{U,h}^0)^2
      +\bigl(h^{2m_{fs}}+k^2\bigr)|\bm U|_m^2
      +\sum_{n=1}^N k
      \norm[H^1(\Osh)]{\xib_{v_s}^{h,n}
        -k^{-1}(\xib_{u,h}^n-\xib_{u,h}^{n-1})}^2\biggr].}
    \end{aligned}
  \end{equation}
  {In the last line, we used the definition of $\CF$ and the fact that
  $1+T\le e^T$.}
  {Using~\eqref{eq:initial-error}, we bound the first term on the right-hand side by $h^{2(m_s+1)}\norm[\H_s]{v_s^0}^2$.
  It remains to estimate the last term.}
  From $\vres = 0$ {{a}nd $v_s^n = \partial_t u(t_n)$, we have}
  \begin{align*}
    \norm[H^1(\Osh)]{\xib_{v_s}^{h,n} &- k^{-1} ({{\xib}_{u,h}^n - \xib_{u,h}^{n-1}})}\\
    &= \norm[H^1(\Osh)]{{{\pi}_h} E_s\partial_t u(t_n) - k^{-1} I_h E_s \delta u^n}\\
&{{\,}\leq \norm[H^1(\Osh)]{I_h E_s\partial_t u(t_n) - k^{-1} I_h E_s \delta u^n}  + \norm[H^1(\Osh)]{(I_h-\pi_h) E_s\partial_t u(t_n)} }
   \end{align*}
   {{F}or the second term, the interpolation estimates~\eqref{proj1} and~\eqref{interpol1}
     combined with the stability of the extension~\eqref{extension_space_stability} yield
   \begin{align*}
   \norm[H^1(\Osh)]{(I_h-\pi_h) E_s\partial_t u(t_n)} \leq ch^{m_s}\norm[H^{m_s+1}(\Omega_s)]{\partial_t u(t_n)}.
   \end{align*}}%
   {{F}or the first term, we obtain using the same stability and interpolation estimates}
     \begin{align*}
\norm[H^1(\Osh)]{I_h E_s\partial_t &u(t_n) - k^{-1} I_h E_s \delta u^n}\\
&\leq{
\norm[H^1(\Osh)]{E_s\partial_t u(t_n) - k^{-1} E_s \delta u^n}}  \\
&\qquad{+ \norm[H^1(\Osh)]{E_s\partial_t u(t_n) - k^{-1} E_s \delta u^n - I_h (E_s\partial_t u(t_n) - k^{-1} E_s \delta u^n)}}
    \\
    &\lesssim k^{-1} \norm[H^1(\Osh)]{k \partial_t u(t_n) - \delta u^n} {+ k^{-1}h^{m_s} \norm[H^{m_s+1}(\Osh)]{k \partial_t u(t_n) - \delta u^n}.}
  \end{align*}
  We proceed as in \eqref{consdtv}-\eqref{consdtv2} to obtain
  \begin{align*}
    k^{-1} \norm[H^1(\Osh)]{k \partial_t u(t_n) - \delta u^n}
    &\lesssim c k^{\frac12} \norm[L^2(I_n,H^1(\Omega_s))]{\partial_t^2  u}\\
  {k^{-1}h^{m_s} \norm[H^{m_s+1}(\Osh)]{k \partial_t u(t_n) - \delta u^n}}
    &{{}\lesssim c k^{\frac12}h^{m_s} \norm[L^2(I_n,H^{m_s+1}(\Omega_s))]{\partial_t^2  u}}
  \end{align*}
  Taking squares and summing up from $n = 1, \dots, N$, this shows that
  \begin{align*}
    \sum_{n=1}^N k
    \norm[H^1(\Osh)]{&\xib_{v_s}^{h,n}
    - k^{-1} ({{\xib}_{u,h}^n - \xib_{u,h}^{n-1}})}^2 \\[-2\jot]
    &\lesssim {k^2 \left(
      \norm[L^2(I,H^1(\Omega_s))]{\partial_t^2 u}^2
      +h^{2m_s}\norm[L^2(I,\H_s)]{\partial_t^2 u}^2\right)}
    {+\sum_{n=1}^N k h^{2m_s}
      \norm[\H_s]{\partial_t u(t_n)}^2} \\
    &{{}\lesssim \bigl(k^2+h^{2m_{fs}}\bigr)|\bm U|_m^2.}
  \end{align*}
\end{proof}

\begin{remark}
  The convergence results of \cref{theo.energyerror} are optimal.
  Taking the square root in \eqref{energy_est}
  gives first-order convergence in time,
  which is optimal for backward Euler,
  and convergence order ${m_{fs}}$ in space,
  which is optimal for $Q_{m_{f}}$-$Q_{m_f-1}$ Taylor--Hood elements in the fluid
  and $Q_{m_{s}}$-$Q_{m_s}$ elements in the solid.
\end{remark}

\section{Numerical example: Lid-driven cavity}
\label{sec_tests}

Finally, we present a numerical experiment
to substantiate our theoretical developments.
The implementation is based on the open-source
finite element library deal.II \cite{dealII94},
in particular step-85 of the tutorial program,
and our recent code developed in \cite{FrKnStWeWi25}.
For the linear systems, we use the parallel sparse solver
MUMPS \cite{Amestoy2001MUMPS}.
In all tests, we use $Q_2$-$Q_1$ Taylor--Hood elements for the fluid
and both $Q_1$-$Q_1$ elements as well as $Q_2$-$Q_2$ elements for the solid,
corresponding to $m_f = 2$ and $m_s \in \set{1,2}$ in the convergence analysis.
As in our previous works \cite{FrKnStWeWi24_ENUMATH,FrKnStWeWi25},
the condition $\vres = 0$ is enforced by adding the equation%
\begin{equation}
\label{eq_uhn_disc}
  \prodOs{u^{h,n} - kv_s^{h,n}}{\psi^h} = \prodOs{u^{h,n-1}}{\psi^h}
  \quad \forall \psi^h \in \U^h
\end{equation}
to the discrete formulation \eqref{eq:definition:Akh},
see also the discussion in Remark~\ref{rem_3_3}.

\subsection{Geometry and parameters}

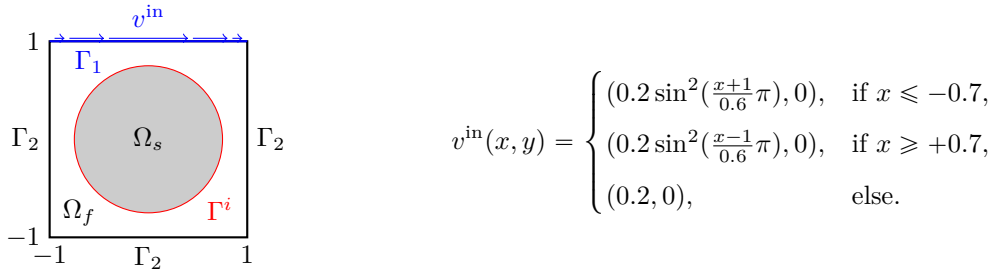
\begin{figure}[tp]
  \centering
  \small
  \begin{tikzpicture}[scale=1.3,baseline=(current bounding box.center)]
    \filldraw[red, fill=black!20, thin] (0,0) circle[radius=0.75];
    \draw[black, thick] (-1,-1) rectangle (-1,-1) rectangle (1,1);
    \draw[blue, thick] (-1,1) -- (1,1);
    \draw[blue, thin, ->] (-0.4,1.03) -- (0.4,1.03);
    \draw[blue, thin, ->] (-0.8,1.03) -- (-0.45,1.03);
    \draw[blue, thin, ->] (-0.95,1.03) -- (-0.85,1.03);
    \draw[blue, thin, ->] (0.45,1.03) -- (0.8,1.03);
    \draw[blue, thin, ->] (0.85,1.03) -- (0.95,1.03);
    \draw (-1,-1) node [below] {$-1$};
    \draw (-1,-1) node [left] {$-1$};
\draw (0,1.05) node [above] {\textcolor{blue}{$v\tsp{in}$}};
    \draw (1,-1) node [below] {$1$};
    \draw (-1,1) node [left] {$1$};
    \draw (0.5,-0.5) node [below right] {\textcolor{red}{$\Gamma^i$}};
    \node at (-0.7,-0.75) {$\Omega_f$};
    \node at (0,0) {$\Omega_s$};
    \node [below, yshift = -0] at (0,-1) {$\Gamma_2$};
    \node [left, xshift = -0] at (-1,0) {$\Gamma_2$};
    \node [below, yshift =  0] at (-0.6,1)  {\textcolor{blue}{$\Gamma_1$}};
    \node [right, xshift =  0] at (1,0)  {$\Gamma_2$};
  \end{tikzpicture}
  \hfil
  \raisebox{-1ex}{$\displaystyle
    v\tsp{in}(x,y)
    =
    \begin{cases}
      (0.2 \sin^2(\frac{x+1}{0.6}\pi), 0), &\text{if } x \le -0.7, \\
      (0.2 \sin^2(\frac{x-1}{0.6}\pi), 0), &\text{if } x \ge +0.7, \\
      (0.2, 0), &\text{else.}
    \end{cases}
    $}%
  \caption{Configuration of
    lid-driven cavity test case with flow profile
    $v\tsp{in}$ at top boundary.
  \label{fig:lid_driven_cavity}}
\end{figure}

We consider the domain $\Omega = (-1,1)^2$
with boundary $\bd\Omega = \Gamma_1 \cup \Gamma_2$,
partitioned into the subdomains \space
$\Omega_s = \defset{x\in\Omega}{\norm[2]{x}^2 < 0.75}$ and
$\Omega_f = \defset{x\in\Omega}{\norm[2]{x}^2 > 0.75}$, see
\cref{fig:lid_driven_cavity}.
We perform a convergence study on a uniform mesh
consisting of quadrilateral cells with initial mesh size $h = 0.25$
using a uniform refinement strategy.

At the upper boundary $\Gamma_1$, we impose the inflow profile $v\tsp{in}$
given in \cref{fig:lid_driven_cavity}.
At $\Gamma_2$, we apply homogeneous Dirichlet conditions
to the fluid velocity $v_f$.
We employ zero initial conditions $v_f^0 = v_s^0 = u^0 = 0$
and increase the inflow gradually by setting
$v\tsp{in}(t,x,y) =
    \tfrac12 (1-\cos(\tfrac{t\pi}{2})) v\tsp{in}(x,y)$ for $t < 2$.
We set {$T = 10$,} $f_f = f_s = 0$ and use the material parameters
$\nu_f = 10^{-3}$, $\rho_f = \rho_s = 1$,
$\mu_s = 5\cdot 10^{-3}$, and $\lambda_s = 10^{-2}$. The solid material parameters are chosen deliberately small, which corresponds to an extremely soft material, in order to obtain a significant deformation of the solid.
The ghost-penalty parameters are chosen as
$\gamma_{v_f} = \gamma_{v_s} = \gamma_u = \gamma_p = 10^{-3}$,
the Nitsche parameter $\gamma_N=10^2$,
and $\wmax = 1$.

\subsection{Numerical Results}

First we study
convergence under mesh refinement in space, see \cref{tab:err_space}.
The time step is fixed to $k = 1.0$.
We {{r}eport} the error with regards to the norms estimated in \cref{theo.energyerror}, i.e., $\normOf{v_f(T)}$,
$\normOs{v_s(T)}$,
$\normOs{\nabla u(T)}$,
\def\smb#1{\smash[b]{#1}}%
$\smb{\norm[{I,\Omega_f}]{\nabla v_f}}$ and
$h\smb{\norm[{I,\Omega_f}]{\nabla p}}$,
where {$\norm[{I,X}]{w}^2:= \smash[t]{\sum_{n=1}^N} k \norm[X]{w(t_n)}^2$}.
The errors are evaluated by comparison against a reference solution
obtained on a finer mesh with $h = 0.0039$.
Convergence orders are estimated numerically
by the errors $e_h$ and $e_{h/2}$
on two consecutive levels with mesh sizes $h$ and $h/2$
by the well-known formula $\alpha = \log_2(e_h / e_{h/2})$.

\begin{table}
	\centering
	\crefformat{theorem}{#2Thm.~#1#3}
	\caption{Error norms w.r.t.\ reference solution computed on mesh with $h = 0.0039$
		and orders of \cref{theo.energyerror}. Top: $(m_f,m_s) = (2,1)$. Bottom: $(m_f,m_s) = (2,2)$.
}
	\label{tab:err_space}
	\sisetup{round-mode=places,round-precision=2}
	\begin{adjustbox}{width=1\textwidth}
		\begin{tabular}{*{12}{>$c<$}}
			\toprule
			&  \multicolumn{2}{c}{$\normOf{v_f(T)}$} & \multicolumn{2}{c}{$\normOs{v_s(T)}$} &  \multicolumn{2}{c}{$\normOs{\nabla u(T)}$}  & \multicolumn{2}{c}{$\norm[{I,\Omega_f}]{\nabla v_f}$} & \multicolumn{2}{c}{$h\norm[{I,\Omega_f}]{\nabla p}$}\\
			\cmidrule(l){2-3} \cmidrule(l){4-5} \cmidrule(l){6-7} \cmidrule(l){8-9} \cmidrule(l){10-11}
			h & \text{error} & \text{order}   & \text{error} & \text{order} & \text{error} & \text{order} & \text{error} & \text{order} & \text{error} & \text{order}\\
			\midrule
			0.2500 & \num{3.14243252e-03} & - & \num{8.13666619e-04} & - & \num{4.35588150e-02} & - & \num{7.15545244e-01} & - & \num{2.27314236e-03} & -\\
			0.1250 & \num{5.81355942e-04} & 2.43 & \num{1.70621901e-04} & 2.25 & \num{2.12200825e-02} & 1.04 & \num{2.96996756e-01} & 1.27 & \num{6.66521046e-04} & 1.77\\
			0.0625 & \num{9.12165556e-05} & 2.67 & \num{3.52184442e-05} & 2.28 & \num{1.07727010e-02} & 0.98 & \num{9.87020506e-02} & 1.59 & \num{1.73216981e-04} & 1.94\\
			0.0312 & \num{1.36732312e-05} & 2.74 & \num{8.31838189e-06} & 2.08 & \num{5.41381349e-03} & 0.99 & \num{2.77392337e-02} & 1.83 & \num{4.64709993e-05} & 1.90\\
			0.0156 & \num{2.18828260e-06} & 2.64 & \num{1.97438011e-06} & 2.07 & \num{2.65684188e-03} & 1.03 & \num{7.18121749e-03} & 1.95 & \num{1.12538228e-05} & 2.05\\
			0.0078 & \num{3.64820973e-07} & 2.58 & \num{4.07454875e-07} & 2.28 & \num{1.19157336e-03} & 1.16 & \num{1.76588410e-03} & 2.02 & \num{2.77973715e-06} & 2.02\\
			\midrule
			\text{\cref{theo.energyerror}} &- & 1.00 &- & 1.00 &- & 1.00 &- & 1.00 &- & 1.00\\
			\midrule
			&  \multicolumn{2}{c}{$\normOf{v_f(T)}$} & \multicolumn{2}{c}{$\normOs{v_s(T)}$} &  \multicolumn{2}{c}{$\normOs{\nabla u(T)}$}  & \multicolumn{2}{c}{$\norm[{I,\Omega_f}]{\nabla v_f}$} & \multicolumn{2}{c}{$h\norm[{I,\Omega_f}]{\nabla p}$}\\
			\cmidrule(l){2-3} \cmidrule(l){4-5} \cmidrule(l){6-7} \cmidrule(l){8-9} \cmidrule(l){10-11}
			h & \text{error} & \text{order}   & \text{error} & \text{order} & \text{error} & \text{order} & \text{error} & \text{order} & \text{error} & \text{order}\\
			\midrule
			0.2500 & \num{3.12844791e-03} & - & \num{6.37506843e-04} & - & \num{1.36013268e-02} & - & \num{7.15562527e-01} & - & \num{9.44915641e-04} & -\\
			0.1250 & \num{5.71871047e-04} & 2.45 & \num{6.74674489e-05} & 3.24 & \num{1.47616419e-03} & 3.20 & \num{2.96989986e-01} & 1.27 & \num{2.00972222e-04} & 2.23\\
			0.0625 & \num{8.74984582e-05} & 2.71 & \num{5.01677253e-06} & 3.75 & \num{2.95903778e-04} & 2.32 & \num{9.87000269e-02} & 1.59 & \num{3.21607368e-05} & 2.64\\
			0.0312 & \num{1.21585177e-05} & 2.85 & \num{3.10459193e-07} & 4.01 & \num{7.12663552e-05} & 2.05 & \num{2.77384328e-02} & 1.83 & \num{6.90789441e-06} & 2.22\\
			0.0156 & \num{1.60712471e-06} & 2.92 & \num{2.37850527e-08} & 3.71 & \num{1.75282686e-05} & 2.02 & \num{7.18083887e-03} & 1.95 & \num{1.55917356e-06} & 2.15\\
			0.0078 & \num{2.11432836e-07} & 2.93 & \num{1.77650602e-09} & 3.74 & \num{4.21502997e-06} & 2.06 & \num{1.76570775e-03} & 2.02 & \num{4.28980661e-07} & 1.86\\
			\midrule
			\text{\cref{theo.energyerror}} &- & 2.00 &- & 2.00 &- & 2.00 &- & 2.00 &- & 2.00\\
			\bottomrule
		\end{tabular}
	\end{adjustbox}
	\crefformat{theorem}{#2Theorem~#1#3}
\end{table}

We observe that all estimated convergence orders
are --on the finest mesh levels-- at least one for $m_s=1$ and at least two for $m_s=2$ (as shown in \cref{theo.energyerror}), or higher.
The difference between linear and quadratic solid elements is most notable in the norms $\normOs{v_s(T)}$, $\normOs{\nabla u(T)}$ and $h\norm[{I,\Omega_f}]{\nabla p}$, whereas the errors for $\norm[{I,\Omega_f}]{\nabla v_f}$ only change from the fifth decimal place onward.
In the norm $\norm[I, \Omega_f]{\nabla v_f}$, we observe slightly reduced
convergence orders on the coarser meshes, while the optimal order of 2 (for $m_s=2$) is attained on the finer meshes.
In the $L^2$-norms of $v_f(T)$ and $v_s(T)$, the convergence seems to be one order larger compared to the estimates in Theorem~\ref{theo.energyerror}. Possibly, this can be shown by means of a duality argument, see \cite{LiSunXieYuOptimalOrder, burman2025dynamicritzprojectionfinite}.

For the convergence analysis in time%
, see \cref{tab:err_space_time},
we choose a uniform mesh of size $h=0.0625$ with $(m_f,m_s)=(2,2)$ and vary the time step $k$,
while all other parameters remain identical.
We observe that the convergence orders
in all norms are reasonably close to or even slightly above one.

\begin{table}
  \centering
  \crefformat{theorem}{#2Thm.~#1#3}
  \caption{Error norms w.r.t\ reference solution computed with time-step $k = 0.0625$
    and orders of \cref{theo.energyerror}.
}
  \label{tab:err_space_time}
  \sisetup{round-mode=places,round-precision=2}
  \begin{adjustbox}{width=1\textwidth}
    \begin{tabular}{*{12}{>$c<$}}
      \toprule
      &  \multicolumn{2}{c}{$\normOf{v_f(T)}$} & \multicolumn{2}{c}{$\normOs{v_s(T)}$} &  \multicolumn{2}{c}{$\normOs{\nabla u(T)}$}  & \multicolumn{2}{c}{$\norm[{I,\Omega_f}]{\nabla v_f}$} & \multicolumn{2}{c}{$h\norm[{I,\Omega_f}]{\nabla p}$}\\
      \cmidrule(l){2-3} \cmidrule(l){4-5} \cmidrule(l){6-7} \cmidrule(l){8-9} \cmidrule(l){10-11}
      k & \text{error} & \text{order}   & \text{error} & \text{order} & \text{error} & \text{order} & \text{error} & \text{order} & \text{error} & \text{order}\\
      \midrule
	  1.0000 & \num{1.47716868e-03} & - & \num{1.88031918e-03} & - & \num{2.01691756e-02} & - & \num{1.67252596e-01} & - & \num{1.56911050e-04} & -\\
	  0.5000 & \num{7.91553946e-04} & 0.90 & \num{1.07125920e-03} & 0.81 & \num{1.32361802e-02} & 0.61 & \num{8.32766465e-02} & 1.01 & \num{7.86089161e-05} & 1.00\\
	  0.2500 & \num{3.72332845e-04} & 1.09 & \num{5.49293195e-04} & 0.96 & \num{7.48673515e-03} & 0.82 & \num{3.75312602e-02} & 1.15 & \num{3.63299479e-05} & 1.11\\
	  0.1250 & \num{1.32581126e-04} & 1.49 & \num{2.20211331e-04} & 1.32 & \num{3.09364378e-03} & 1.28 & \num{1.29056815e-02} & 1.54 & \num{1.28208583e-05} & 1.50\\
      \midrule
      \text{\cref{theo.energyerror}} &- & 1.00 &- & 1.00 &- &1.00 &- & 1.00 &- & 1.00\\
      \bottomrule
    \end{tabular}
  \end{adjustbox}
  \crefformat{theorem}{#2Theorem~#1#3}
\end{table}

\section{Conclusion}
In this work, we established stability and optimal-order a priori error estimates
  for fully Eulerian fluid-structure interaction with a fixed interface.
  While for the temporal discretization
  we employed the standard backward Euler scheme,
  the spatial discretization was based on an unfitted finite element method
  with inf-sup stable element pairs.
  The main novelties compared to previous work include
  a proof of stability
  with respect to the computational domains
  for $Q_r$ finite elements using ghost penalties,
  and --in the context of fluid-structure interactions--
  the consideration of Taylor-Hood finite elements in the fluid and of unfitted finite elements and ghost-penalties
  in the solid domain.
  Under the inverse CFL condition $h^2 \lesssim k$, we could show optimal-order error estimates,
  both in time and space.
  Finally, we presented numerical computations that substantiate the analytical convergence orders.
  Future research might focus on the derivation of $L^2(L^2)$-error estimates for the pressure variable, the analysis of different time-stepping schemes,
  such as backward difference formulas (BDF), as well as time-dependent interfaces. Numerically, the proposed unfitted fully Eulerian formulation could be applied to FSI with large deformations or contact.

\section*{Acknowledgments}
A.-K. Wenske and M. C. Steinbach gratefully acknowledge financial support from
the German Research Foundation (DFG) -- SFB1463 -- 434502799.
S. Frei and T. Wick gratefully acknowledge financial support from the DFG
with the grant number 548064929.

\bibliographystyle{plainnat}
\bibliography{shorttitles,lit,references-supplement}

\end{document}